\documentclass{amsart}

\usepackage{amsmath}
\usepackage{amssymb}
\usepackage{amsthm}
\usepackage{mathrsfs}

\DeclareMathOperator{\Gal}{Gal}

\DeclareMathOperator{\Rep}{Rep}

\newtheorem{thm}{Theorem}[section]
\newtheorem{prop}[thm]{Proposition}
\newtheorem{defn}[thm]{Definition}
\newtheorem{lem}[thm]{Lemma}

\newtheorem{rem}[thm]{Remark}

\newtheorem*{thmn}{Theorem}

\newcommand{\F}{\mathbb{F}}

\begin{document}

\title
[Ramification and moduli spaces of finite flat models]%
{Ramification and moduli spaces\\
of finite flat models}
\author{Naoki Imai}
\address{Research Institute for Mathematical Sciences, 
Kyoto University, Kyoto 606-8502 Japan}
\email{naoki@kurims.kyoto-u.ac.jp}


\maketitle

\begin{abstract}
We determine the type of the zeta functions 
and the range of the dimensions 
of the moduli spaces of finite flat models 
of two-dimensional local Galois representations over finite fields. 
This gives a generalization of 
Raynaud's theorem on the uniqueness of 
finite flat models in low ramifications. 
\end{abstract}

\section*{Introduction}
Let $K$ be a finite extension of the 
field $\mathbb{Q}_p$ of $p$-adic numbers. 
We assume $p>2$. 
Let $e$ be the ramification index of $K$ over $\mathbb{Q}_p$, 
and $k$ be the residue field of $K$. 
We consider a two-dimensional continuous 
representation $V_{\F}$ 
of the absolute Galois group $G_K$ 
over a finite field $\F$ of characteristic $p$. 
By a finite flat model of $V_{\F}$, 
we mean a finite flat group scheme $\mathcal{G}$ 
over $\mathcal{O}_K$, 
equipped with an action of $\F$, 
and an isomorphism 
$V_{\F} \stackrel{\sim}{\longrightarrow} \mathcal{G}(\overline{K})$ 
that respects the action of $G_K$ and $\F$. 
We assume that $V_{\F}$ has at least one finite flat model. 
If $e < p-1$, the finite flat model of $V_{\F}$ is 
unique by Raynaud's result \cite[Theorem 3.3.3]{Ray}.
In general, there are finitely many finite flat models 
of $V_{\F}$, and these appear as the $\F$-rational points of 
the moduli space of finite flat models of $V_{\F}$, 
which we denote by $\mathscr{GR}_{V_{\F},0}$.
It is natural to ask about the dimension
of $\mathscr{GR}_{V_{\F},0}$.
In this paper, 
we determine the type of the zeta functions 
and the range of the dimensions 
of the moduli spaces. 
The main theorem is the following. 

\begin{thmn}
Let 
$d_{V_{\F}} =\dim \mathscr{GR}_{V_{\F},0}$, 
and 
$Z(\mathscr{GR}_{V_{\F},0} ;T)$ 
be the zeta function of 
$\mathscr{GR}_{V_{\F},0}$. 
We put $n=[k:\F_p ]$. 
Then followings are true. 
\begin{enumerate}
\item 
After extending the field $\F$ sufficiently, 
we have 
\[
 Z(\mathscr{GR}_{V_{\F},0} ;T) = 
 \prod_{i=0} ^{d_{V_{\F}}} 
 (1-|\F|^i T)^{-m_i} 
\]
for some $m_i \in \mathbb{Z}$ such that $m_{d_{V_{\F}}} >0$. 
\item 
If $n=1$, 
we have 
\[
 0 \leq 
 d_{V_{\F}} \leq  
 \biggl[ \frac{e+2}{p+1} \biggr]. 
\]
If $n\geq 2$, 
we have 
\[
 0 \leq 
 d_{V_{\F}} \leq 
 \biggl[ \frac{n+1}{2} \biggr]
 \biggl[ \frac{e}{p+1} \biggr]+ 
 \biggl[ \frac{n-2}{2} \biggr] 
 \biggl[ \frac{e+1}{p+1} \biggr]+ 
 \biggl[ \frac{e+2}{p+1} \biggr]. 
\]
Here, $[x]$ is the greatest integer 
less than or equal to $x$ for 
$x \in \mathbb{R}$.

Furthermore, each equality in the above inequalities 
can happen for any finite extension $K$ of $\mathbb{Q}_p$. 
\end{enumerate}
\end{thmn} 

Raynaud's result says that if $e<p-1$ then 
$\mathscr{GR}_{V_{\F},0}$ is one point, 
that is, zero-dimensional and connected. 
If $e < p-1$, the above theorem also implies that 
$\mathscr{GR}_{V_{\F},0}$ is zero-dimensional. 
So it gives a dimensional generalization 
of Raynaud's result 
for two-dimensional Galois representations. 
The connectedness of 
$\mathscr{GR}_{V_{\F},0}$ is 
completely false in general. 
For example, we can check that 
if $K=\mathbb{Q}_p (\zeta_p )$ and 
$V_{\F}$ is trivial representations 
then $\mathscr{GR}_{V_{\F},0}$ consists of 
$\mathbb{P}_{\F} ^1$ and two points 
(c.f. \cite[Proposition 2.5.15(2)]{Kis}). 
Here $\mathbb{P}_{\F} ^1$ denotes 
the $1$-dimensional projective space 
over $\F$. 

In the section $1$, we recall 
the moduli space of finite flat models, 
and give some Lemmas. 
We also give an example for any $K$ where 
the moduli space of finite flat models 
is one point. 

A proof of the main theorem separates into two cases, 
that is, the case where $V_{\F}$ is not absolutely irreducible 
and the case where $V_{\F}$ is absolutely irreducible. 
In section $2$, we treat 
the case where $V_{\F}$ is not absolutely irreducible. 
In this case, we decompose $\mathscr{GR}_{V_{\F},0}$ 
into affine spaces in the level of rational points. 
Then we express the dimensions of these affine spaces explicitly 
and bound it by combinatorial arguments. 
In section $3$, we treat 
the case where $V_{\F}$ is absolutely irreducible. 
A proof is similar to the case 
where $V_{\F}$ is not absolutely irreducible, 
but, in this case, 
we have to decompose $\mathscr{GR}_{V_{\F},0}$ into 
$\mathbb{A}_{\F} ^d$ and 
$\mathbb{A}_{\F} ^{d-1} \times \mathbb{G}_m$ and 
$\mathbb{A}_{\F} ^{d-2} \times \mathbb{G}_m ^2$ 
in the level of rational points. 
Here $\mathbb{A}_{\F} ^d$ denotes the 
$d$-dimensional affine space over $\F$, 
and $\mathbb{G}_m$ is $\mathbb{A}_{\F} ^1 -\{0\}$. 

In the section $4$, we state the main theorem 
and prove it by collecting the results of 
former sections. 

\subsection*{Acknowledgment}
The author is grateful to his advisor Takeshi Saito 
for his careful reading of an earlier version of this paper 
and for his helpful comments. 
He would like to thank the referees 
for their careful reading of this paper 
and a number of suggestions for improvements. 
\subsection*{Notation}
Throughout this paper, we use the following notation.
Let $p>2$ be a prime number, 
and $k$ be the finite field 
of cardinality $q=p^n$. 
The Witt ring of $k$ is denoted by $W(k)$. 
Let $K_0$ be the quotient field of $W(k)$, 
and $K$ be 
a totally ramified extension of $K_0$ of degree $e$. 
The ring of integers of $K$ is denoted by $\mathcal{O}_K$, 
and the absolute Galois group of $K$ is 
denoted by $G_K$. 
Let $\F$ be a finite field of characteristic $p$.
For a ring $A$, 
the formal power series ring of $u$ over $A$ is
denoted by $A [[u]]$, 
and we put $A ((u))=A [[u]](1/u)$. 
For a field $F$, 
the algebraic closure of $F$ 
is denoted by $\overline{F}$ and 
the separable closure of $F$ 
is denoted by $F^{\mathrm{sep}}$. 
Let $v_u$ be the valuation of 
$\F ((u))$ normalized by $v_u (u)=1$, 
and we put $v_u (0)=\infty$.
For $x \in \mathbb{R}$, 
the greatest integer less than or equal to 
$x$ is denoted by $[x]$. 
For a positive integer $d$, 
the $d$-dimensional affine space over $\F$ is 
denoted by $\mathbb{A}_{\F} ^d$. 
Let $\mathbb{G}_m$ be $\mathbb{A}_{\F} ^1 -\{0\}$. 

\section{Preliminaries}
First of all, we recall the moduli spaces of 
finite flat models constructed by Kisin in \cite{Kis}. 

Let $V_{\F}$ be a continuous 
two-dimensional representation of $G_K$ over $\F$. 
We assume that $V_{\F}$
comes from the generic fiber of 
a finite flat group scheme over $\mathcal{O}_K$.
The moduli space of finite flat models of $V_{\F}$, 
which is denoted by $\mathscr{GR}_{V_{\F},0}$, 
is a projective scheme over $\F$. 
An important property of $\mathscr{GR}_{V_{\F},0}$ 
is the following Proposition.

\begin{prop}\label{property}
For any finite extension $\F'$ of $\F$, 
there is a natural bijection between 
the set of isomorphism classes of finite flat models 
of $V_{\F'} =V_{\F} \otimes _{\F} \F'$ 
and $\mathscr{GR}_{V_{\F},0} (\F')$.
\end{prop}

\begin{proof}
This is \cite[Corollary 2.1.13]{Kis}.
\end{proof}

Let $\mathfrak{S} = W(k)[[u]] $, and $\mathcal{O}_{\mathcal{E}}$ be 
the $p$-adic completion of $\mathfrak{S} [1/u]$. 
There is an $p$-adically continuous 
action of $\phi$ on 
$\mathcal{O}_{\mathcal{E}}$ 
determined by Frobenius on $W(k)$ 
and $u \mapsto u^p$. 
We fix a uniformizer $\pi$ of $\mathcal{O}_K$, 
and choose elements $\pi _m \in \overline{K}$ such that
$\pi _0 = \pi$ and $\pi ^p _{m+1} = \pi _m$ 
for $m\geq 0$, and 
put $K_{\infty} = \bigcup _{m\geq 0} K(\pi _m) $. 

Let $\Phi {\mathrm{M}}_{\mathcal{O}_{\mathcal{E}},\F}$ 
be the category of finite 
$(\mathcal{O}_{\mathcal{E}} \otimes_{\mathbb{Z}_p} \F)$-modules 
$M$ equipped with $\phi$-semi-linear map 
$M \to M$ such that the induced 
$(\mathcal{O}_{\mathcal{E}} \otimes_{\mathbb{Z}_p} \F)$-linear 
map $\phi ^* (M) \to M$ is an isomorphism. 
Let $\Rep_{\F} (G_{K_{\infty}} )$ be the 
category of continuous representations of $G_{K_{\infty}}$ over $\F$. 
Then the functor 
\[
 T: \Phi {\mathrm{M}}_{\mathcal{O}_{\mathcal{E}},\F} 
 \to \Rep_{\F} (G_{K_{\infty}});\ 
 M \mapsto \bigl( k((u))^{\mathrm{sep}} \otimes_{k((u))} M \bigr)^{\phi =1} 
\]
gives an equivalence of abelian categories 
as in \cite[(1.1.12)]{Kis}. 
Here $\phi$ acts on $k((u))^{\mathrm{sep}}$ 
by the $p$-th power map. 
We take the $\phi$-module 
$M_{\F} \in \Phi {\mathrm{M}}_{\mathcal{O}_{\mathcal{E}},\F}$ 
such that $T(M_{\F})$ is isomorphic to 
$V_{\F} (-1)|_{G_{K_{\infty}}}$. 
Here $(-1)$ denotes the inverse of the Tate twist. 

The moduli space $\mathscr{GR}_{V_{\F},0}$ is described via 
the Kisin modules as in the following.

\begin{prop}\label{description}
For any $\F$-algebra $A$, 
the elements of $\mathscr{GR}_{V_{\F},0}(A)$ 
naturally correspond to finite projective 
$(k[[u]] \otimes _{\F_p} A)$-submodules 
$\mathfrak{M}_A \subset M_\F \otimes _{\F} A$ 
that satisfy the followings: 
\begin{enumerate}
\item 
$\mathfrak{M}_A$ generates 
$M_\F \otimes _{\F} A$ over $k((u)) \otimes _{\F_p} A$. 
\item 
$u^e \mathfrak{M}_A \subset 
(1\otimes \phi ) \bigl( \phi ^* (\mathfrak{M}_A) \bigr)
\subset \mathfrak{M}_A$.
\end{enumerate}
\end{prop}

\begin{proof}
This follows from the construction 
of $\mathscr{GR}_{V_{\F},0}$ in 
\cite[Corollary 2.1.13]{Kis}. 
\end{proof}

By Proposition \ref{description}, 
we often identify a point of 
$\mathscr{GR}_{V_{\F},0} (\F')$ with 
the corresponding finite free 
$k[[u]] \otimes _{\F_p} \F'$-module. 

 From now on, we assume $\F_{q^2} \subset \F$ and 
fix an embedding $k \hookrightarrow \F$. 
This assumption does not matter, 
because we may extend $\F$ to prove the main theorem. 
We consider the isomorphism
\[
 \mathcal{O}_{\mathcal{E}} \otimes _{\mathbb{Z}_p} \F \cong 
 k((u))\otimes _{\F _p} \F \stackrel{\sim}{\to}
 \prod _{\sigma \in \Gal (k/{\F_p})} \F ((u))\ ;\ 
 \Bigl( \sum a_i u^i \Bigr) \otimes b \mapsto
 \Bigl( \sum \sigma (a_i ) b u^i \Bigr) _{\sigma}
\]
and let $\epsilon_{\sigma} \in k((u)) \otimes _{\F_p} \F $ be 
the primitive idempotent corresponding to $\sigma$.
Take $\sigma _1 , \cdots , \sigma _n \in \Gal (k/\F _p)$ such 
that $\sigma _{i+1} = \sigma _i \circ \phi ^{-1}$. 
Here we regard $\phi$ as the $p$-th power Frobenius, 
and use the convention that 
$\sigma _{n+i} =\sigma _i$.
In the following, we often use such conventions.
Then we have 
$\phi (\epsilon _{\sigma _i}) = \epsilon _{\sigma _{i+1}}$ and 
$\phi : M_{\F} \to M_{\F}$ determines 
$\phi :\epsilon _{\sigma _i} M_{\F} \to 
 \epsilon _{\sigma _{i+1}} M_{\F}$.
For $(A_i )_{1 \leq i \leq n} \in GL_2 \bigl( \F ((u)) \bigr)^n$, 
we write
\[
 M_{\F} \sim (A_1 ,A_2 , \ldots ,A_n )=(A_i)_i
\]
if there is a basis $\{ e^i _1 ,e^i _2 \}$ 
of $\epsilon _{\sigma _i} M_{\F}$ over $\F ((u))$ 
such that 
$\phi 
\begin{pmatrix} e^i _1 \\ e^i _2 \end{pmatrix}
 = A_i 
\begin{pmatrix} e^{i+1} _1 \\ e^{i+1} _2 \end{pmatrix}$. 
We use the same notation for any sublattice 
$\mathfrak{M} _{\F} \subset M_{\F}$ similarly. 
Here and in the following, we consider only sublattices 
that are $(\mathfrak{S}\otimes _{\mathbb{Z}_p} \F)$-modules.

Let $A$ be an $\F$-algebra, and 
$\mathfrak{M}_A$ be a finite free 
$(k[[u]] \otimes _{\F_p} A)$-submodules 
of $M_\F \otimes _{\F} A$ 
that generate 
$M_\F \otimes _{\F} A$ over $k((u)) \otimes _{\F_p} A$. 
We choose a basis $\{ e^i _1 ,e^i _2 \}_i$ of 
$\mathfrak{M}_A$ over $k[[u]] \otimes _{\F_p} A$. 
For 
$B=(B_i )_{1 \leq i \leq n} \in 
 GL _2 \bigl(\F((u)) \otimes_{\F_p} A \bigr)^n$, 
the $(\mathfrak{S}\otimes _{\mathbb{Z}_p} A)$-module 
generated by the entries of 
$\biggl{\langle} B_i 
\begin{pmatrix} e^i _1 \\ e^i _2 \end{pmatrix} 
\biggr{\rangle}$ for 
$1 \leq i \leq n$
with the basis given by these entries 
is denoted by 
$B\cdot \mathfrak{M} _{A}$. 
Note that $B\cdot \mathfrak{M} _A$ depends on 
the choice of the basis of $\mathfrak{M} _A$. 
We can see that 
if $\mathfrak{M}_{\F} \sim (A_i)_i$ for 
$(A_i )_{1 \leq i \leq n} \in GL_2 \bigl( \F((u)) \bigr)^n$ 
with respect to a given basis, 
then we have 
\[
 B\cdot \mathfrak{M} _{\F} \sim 
 \bigl( \phi (B_i )A_i (B_{i+1} )^{-1} 
 \bigr)_i 
\]
with respect to the induced basis. 

\begin{lem}\label{equiv}
Suppose $\F'$ is a finite extension of $\F$, 
and $x \in \mathscr{GR}_{V_{\F},0}(\F')$ 
corresponds to 
$\mathfrak{M} _{\F'}$. 
Put 
$\mathfrak{M} _{j,\F'} = 
 \Biggl(
 \begin{pmatrix}
 u^{s_{j,i}} & v_{j,i} \\ 0 & u^{t_{j,i}}
 \end{pmatrix} 
 \Biggr)_i \cdot
 \mathfrak{M} _{\F'}$
for 
$1\leq j \leq 2$, $s_{j,i} ,t_{j,i} \in \mathbb{Z}$ and 
$v_{j,i} \in \F'((u))$.
Assume $\mathfrak{M} _{1,\F'}$ and $\mathfrak{M} _{2,\F'}$ 
correspond to $x_1 , x_2 \in \mathscr{GR}_{V_{\F},0}(\F')$ 
respectively. 
Then $x_1 =x_2$ if and only if
\[
 s_{1,i} =s_{2,i} ,\ t_{1,i} =t_{2,i} 
 \textrm{ and } 
 v_{1,i} -v_{2,i} \in u^{t_{1,i}} \F'[[u]] 
 \textrm{ for all } i.
\]
\end{lem}
\begin{proof}
The equality $x_1 =x_2$ is equivalent to 
the existence of 
$B=(B_i )_{1 \leq i \leq n} \in GL_2 (\F'[[u]])^n$ 
such that 
\[
B_i 
 \begin{pmatrix}
  u^{s_{1,i}} & v_{1,i} \\ 0 & u^{t_{1,i}} 
 \end{pmatrix} =
 \begin{pmatrix}
  u^{s_{2,i}} & v_{2,i} \\ 0 & u^{t_{2,i}} 
 \end{pmatrix} 
\]
for all $i$. 
It is further equivalent to the condition that 
\[
 \begin{pmatrix}
 u^{s_{2,i} -s_{1,i}} 
 & v_{2,i} u^{-t_{1,i}} - u^{s_{2,i} -s_{1,i} -t_{1,i}} v_{1,i} \\
 0 & u^{t_{2,i} -t_{1,i}}
 \end{pmatrix} 
 \in GL_2 (\F'[[u]]) 
\] 
for all $i$. 
The last condition is equivalent to the desired condition. 
\end{proof}

\begin{lem}\label{stracture}
Suppose $V_{\F}$ is absolutely irreducible. 
If $\F'$ is the quadratic extension of $\F$, 
then 
\[
 M_{\F} \otimes _{\F} \F ' \sim 
\Biggl(
\begin{pmatrix}
 0 & \alpha _1 \\ \alpha _1 u^m & 0
\end{pmatrix}
,
\begin{pmatrix}
 \alpha _2 & 0 \\ 0 & \alpha _2
\end{pmatrix}
,
\ldots
,
\begin{pmatrix}
 \alpha _n & 0 \\ 0 & \alpha _n
\end{pmatrix}
\Biggr)
\]
for some $\alpha _i \in (\F')^{\times}$ and a positive integer $m$ 
such that $(q+1) \nmid m$. 
Conversely, for each positive integer $m$ 
such that $(q+1) \nmid m$, 
there exists an absolutely irreducible representation 
$V_{\F}$ as above. 
\end{lem}

\begin{proof}
The first statement is \cite[Lemma 1.2]{Ima}, 
and the second statement follows from the proof 
of \cite[Lemma 1.2]{Ima}. 
We have used the assumption 
$\F_{q^2} \subset \F$ in this Lemma. 
\end{proof}

\begin{prop}\label{onepoint}
If 
$M_{\F} \sim 
\Biggl(
\begin{pmatrix}
 u^e & u \\ 0 & 1 
\end{pmatrix}
\Biggr)_i$, 
then $\mathscr{GR}_{V_{\F},0} (\F')$ 
is one point for any finite extension 
$\F'$ of $\F$. 
\end{prop}

\begin{proof}
Let $\mathfrak{M} _{0,\F}$ be the 
lattice of $M_{\F}$ generated by 
the basis giving 
\[
 M_{\F} \sim 
 \Biggl(
 \begin{pmatrix}
  u^e & u \\ 0 & 1 
 \end{pmatrix}
 \Biggr)_i ,
\] 
and let 
$\mathfrak{M} _{0,\F'} =
\mathfrak{M} _{0,\F} \otimes _{\F} \F'$ 
for finite extensions $\F'$ of $\F$. 
Then $\mathfrak{M} _{0,\F'}$ 
gives a point of 
$\mathscr{GR}_{V_{\F},0} (\F')$. 
By the Iwasawa decomposition, 
any point $\mathfrak{M} _{\F'}$ 
of 
$\mathscr{GR}_{V_{\F},0} (\F')$ 
is written as 
$\Biggl(
\begin{pmatrix}
 u^{-s_i} & v_i \\ 0 & u^{t_i}
\end{pmatrix}
\Biggr)_i \cdot 
\mathfrak{M} _{0,\F'}$ 
for 
$s_i ,t_i \in \mathbb{Z}$ and 
$v_i \in \F((u))$. 
Then we have 
\begin{align*}
 \mathfrak{M} _{\F'} \sim &
 \Biggl(
 \begin{pmatrix}
  u^{-ps_i} & \phi (v_i) \\ 0 & u^{pt_i}
 \end{pmatrix}
 \begin{pmatrix}
  u^e & u \\ 0 & 1 
 \end{pmatrix}
 \begin{pmatrix}
  u^{s_{i+1}} & -v_{i+1} u^{s_{i+1} -t_{i+1} } \\ 
  0 & u^{-t_{i+1} }
 \end{pmatrix}
 \Biggr)_i \\
 &=
 \Biggl(
 \begin{pmatrix}
  u^{e-ps_i +s_{i+1}} & 
  u^{1-ps_i -t_{i+1} }+\phi (v_i )u^{-t_{i+1} } 
  -v_{i+1} u^{e-ps_i +s_{i+1} -t_{i+1} } \\ 
  0 & u^{pt_i -t_{i+1} }
 \end{pmatrix}
 \Biggr)_i
\end{align*}
with respect to the basis induced 
from the given basis of $\mathfrak{M} _{0,\F'}$. 
We put $r_i =-v_u (v_i)$. 

By 
$u^e \mathfrak{M}_{\F'} \subset 
(1\otimes \phi ) \bigl( \phi ^* (\mathfrak{M}_{\F'}) \bigr)
\subset \mathfrak{M}_{\F'}$, 
we have 
$e-ps_i +s_{i+1} \leq e$ and 
$pt_i -t_{i+1} \geq 0$ for all $i$, 
so we get $s_i ,t_i \geq 0$ for all $i$. 

We are going to show that $1-ps_i -t_{i+1} \geq 0$ for all $i$. 
We assume that $1-ps_{i_0} -t_{i_0 +1} <0$ for some $i_0$. 
Then 
$v_u (v_{i_0 +1} u^{e-ps_{i_0} +s_{i_0 +1} -t_{i_0 +1} })
\leq 1-ps_{i_0} -t_{i_0 +1}$, 
because $\phi (v_{i_0} )u^{-t_{i_0 +1} }$ has 
no term of degree $1-ps_{i_0} -t_{i_0 +1}$. 
So we get $r_{i_0 +1} -s_{i_0 +1} \geq e-1 \geq 0$. 
Take an index $i_1$ such that 
$r_{i_1} -s_{i_1}$ 
is the maximum. 
We note that $r_{i_1} -s_{i_1} \geq 0$. 
Then we have 
$v_u \bigl( \phi (v_{i_1} )u^{-t_{i_1 +1}}\bigr) =
 v_u (v_{i_1 +1} u^{e-ps_{i_1} +s_{i_1 +1} -t_{i_1 +1} })$, 
because 
$v_u \bigl( \phi (v_{i_1} )u^{-t_{i_1 +1}} \bigr) 
 \leq -ps_{i_1} -t_{i_1 +1}$. 
So we get 
$r_{i_1 +1} -s_{i_1 +1} = p(r_{i_1} -s_{i_1} )+e 
 > r_{i_1} -s_{i_1}$. 
This is a contradiction. 
Thus we have proved that 
$1-ps_i -t_{i+1} \geq 0$ for all $i$, 
and this is equivalent to that
$s_i =0$ and $0 \leq t_i \leq 1$ for all $i$. 

We assume $t_i =1$ for some $i$. 
Then we have $t_i =1$ for all $i$, because 
$pt_{i-1} -t_i \geq 0$ for all $i$.
We show that $r_i \leq -1$ for all $i$. 
We take an index $i_2$ such that 
$r_{i_2}$ is the maximum, 
and assume that 
$r_{i_2} \geq 0$.
Then we have $r_{i_2 +1} =pr_{i_2} +e>r_{i_2}$, 
because 
$v_u \bigl(1+\phi (v_{i_2})u^{-1} -v_{i_2 +1} u^{e-1} \bigr) \geq 0$. 
This is a contradiction. 
So we have $r_i \leq -1$ for all $i$.
Then we may assume $v_i =0$ for all $i$ 
by Lemma \ref{equiv}. 
Now we have 
$\mathfrak{M} _{\F'} \sim 
 \Biggl(
 \begin{pmatrix}
  u^e & 1 \\ 0 & u^{p-1} 
 \end{pmatrix}
 \Biggr)_i$, 
but this contradicts 
$u^e \mathfrak{M}_{\F'} \subset 
 (1\otimes \phi ) \bigl( \phi ^* (\mathfrak{M}_{\F'}) \bigr)$.

Thus we have proved $s_i =t_i =0$ for all $i$. 
Then we have $r_i \leq 0$, 
because $v_u (u+\phi(v_i )-v_{i+1} u^e) \geq 0$.
So we may assume $v_i =0$ for all $i$ 
by Lemma \ref{equiv}, 
and we have 
$\mathfrak{M}_{\F'} =\mathfrak{M}_{0,\F'}$. 
This shows that $\mathscr{GR}_{V_{\F},0} (\F')$ 
is one point. 
\end{proof} 

\section{The case where $V_{\F}$ is not absolutely irreducible}

In this section, 
we give the maximum of 
the dimensions of the moduli spaces 
in the case where 
$V_{\F}$ is not absolutely irreducible. 
We put 
$d_{V_{\F}} = 
 \dim \mathscr{GR}_{V_{\F},0}$. 
In the proof of the following Proposition, 
three Lemmas appear. 

\begin{prop}\label{reducible}
We assume $V_{\F}$ is not absolutely irreducible, 
and write $e=(p+1)e_0 +e_1$ for $e_0 \in \mathbb{Z}$ and 
$0 \leq e_1 \leq p$. 
Then the followings are true. 
\begin{enumerate}
\item

There are $m_i \in \mathbb{Z}$ 
for $0 \leq i \leq d_{V_{\F}}$ 
such that $m_i \geq 0$, 
$m_{d_{V_{\F}}} >0$ and 
\[
 |\mathscr{GR}_{V_{\F},0} (\F')| =
 \sum_{i=0} ^{d_{V_{\F}}} 
 m_i |\F'|^i 
\]
for all sufficiently large 
extensions $\F'$ of $\F$. 
\item 
\begin{enumerate}
\item 
In the case $0 \leq e_1 \leq p-2$, 
we have $d_{V_{\F}} \leq ne_0$. 
In this case, if 
\[
 M_{\F} \sim
 \Biggl( 
 \begin{pmatrix}
  u^{e_0} & 0 \\ 0 & u^{pe_0}
 \end{pmatrix}
 \Biggr)_i , 
\]
then $d_{V_{\F}} = ne_0$. 
\item 
In the case $e_1 = p-1$, 
we have $d_{V_{\F}} \leq ne_0 +1$. 
In this case, if 
\[
 M_{\F} \sim
 \Biggl( 
 \begin{pmatrix}
  u^{e_0} & 0 \\ 0 & u^{pe_0 +p-1}
 \end{pmatrix}
 \Biggr)_i , 
\] 
then $d_{V_{\F}} = ne_0 +1$. 
\item 
In the case $e_1 = p$, 
we have 
$d_{V_{\F}} \leq 
 ne_0 +\max \bigl\{ [n/2],1 \bigr\}$.
In this case, if $n=1$ and 
\[
 M_{\F} \sim
 \begin{pmatrix}
  u^{e_0} & 0 \\ 0 & u^{pe_0 +p-1}
 \end{pmatrix}, 
\]
then $d_{V_{\F}} = e_0 +1$, 
and if $n \geq 2$ and 
\[
 M_{\F} \sim 
 \Biggl(
 \begin{pmatrix}
  u^{e_{0,i} } & 0 \\ 0 & u^{p(2e_0 +1 -e_{0,i} )} 
 \end{pmatrix} 
 \Biggr)_i , 
\]
then $d_{V_{\F}} = ne_0 +[n/2]$. 
Here, $e_{0,i} =e_0$ 
if $i$ is odd, 
and $e_{0,i} =e_0 +1$ 
if $i$ is even. 
\end{enumerate}
\end{enumerate}
\end{prop}

\begin{proof}
Extending the field $\F$, we may assume that 
$V_{\F}$ is reducible. 
Let $\mathfrak{M} _{0,\F}$ be a lattice of 
$M_{\F}$ corresponding to a point of 
$\mathscr{GR}_{V_{\F},0} (\F)$. 
Then we take and fix a basis of 
$\mathfrak{M} _{0,\F}$ over 
$k[[u]] \otimes _{\F_p} \F$ 
such that 
$\mathfrak{M} _{0,\F} \sim 
\Biggl(
\begin{pmatrix} 
\alpha_i u^{a_{0,i}} & w_{0,i} \\ 0 & \beta_i u^{b_{0,i}} 
\end{pmatrix}
\Biggr)_i$ 
for 
$\alpha_i ,\beta_i \in \F^{\times}$, 
$0 \leq a_{0,i} ,b_{0,i} \leq e$ and 
$w_{0,i} \in \F[[u]]$. 
For any finite extension $\F'$ of $\F$, 
we put 
$\mathfrak{M} _{0,\F'} =
\mathfrak{M} _{0,\F} \otimes _{\F} \F'$ 
and 
$M_{\F'} =M_{\F} \otimes _{\F} \F'$. 
By the Iwasawa decomposition, 
any sublattice of $M_{\F'}$ can be written as 
$\Biggl(\begin{pmatrix}
  u^{s_i} & v_i ' \\ 0 & u^{t_i} 
 \end{pmatrix}
 \Biggr)_i \cdot 
 \mathfrak{M} _{0,\F'}$
for $s_i ,t_i  \in \mathbb{Z}$ and $v_i ' \in \F'((u))$. 

We put 
\[
 I=\bigl\{
 (\underline{a},\underline{b}) 
 \in \mathbb{Z}^n \times \mathbb{Z}^n \bigm| 
 \underline{a}=(a_i )_{1 \leq i \leq n},\  
 \underline{b}=(b_i )_{1 \leq i \leq n},\ 
 0 \leq a_i ,b_i \leq e
 \bigr\}, 
\]
and 
\begin{align*}
 \mathscr{GR}_{V_{\F},0,\underline{a},\underline{b}}(\F') &=
 \Biggl\{
 \Biggl(
 \begin{pmatrix}
 u^{s_i } & v_i ' \\ 0 & u^{t_i } 
 \end{pmatrix} 
 \Biggr)_i \cdot 
 \mathfrak{M} _{0,\F'}
 \in
 \mathscr{GR}_{V_{\F},0}(\F') 
 \Biggm| 
 s_i ,t_i \in \mathbb{Z}, 
 v_i ' \in \F'((u)), \\ 
 & \hspace*{10em}
 a_i = a_{0,i} +ps_i -s_{i+1} ,\ b_i =b_{0,i} +pt_i -t_{i+1} 
\Biggr\} 
\end{align*}
for 
$(\underline{a},\underline{b}) =
 \bigl( (a_i )_{1 \leq i \leq n}, 
 (b_i )_{1 \leq i \leq n} \bigr) \in I$. 
Then we have 
\[
 \mathscr{GR}_{V_{\F},0}(\F') =
 \bigcup_{(\underline{a},\underline{b}) \in I} 
 \mathscr{GR}_{V_{\F},0,\underline{a},\underline{b}}(\F'), 
\]
and this is a disjoint union 
by Lemma \ref{equiv}.

Take 
$\mathfrak{M} _{\F'} = 
 \Biggl(
 \begin{pmatrix}
  u^{s_i } & v_i ' \\ 0 & u^{t_i } 
 \end{pmatrix}
 \Biggr)_i \cdot 
 \mathfrak{M} _{0,\F'} 
 \in \mathscr{GR}_{V_{\F},0,\underline{a},\underline{b}}(\F')$ 
with the basis 
induced from the basis of $\mathfrak{M} _{0,\F'}$, 
then 
$\mathfrak{M} _{\F'} \sim 
\Biggl(
\begin{pmatrix}
 \alpha_i u^{a_i} & w_i \\ 0 & \beta_i u^{b_i}
\end{pmatrix}
\Biggr)_i$ 
for some 
$(w_i )_{1 \leq i \leq n} \in \F'[[u]]^n$. 
We note that 
$a_i +b_i -v_u (w_i)\leq e$ 
for all $i$ by the condition 
$u^e \mathfrak{M}_{\F'} \subset 
(1\otimes \phi ) \bigl( \phi ^* (\mathfrak{M}_{\F'}) \bigr)$. 

Now, any 
$\mathfrak{M}' _{\F'} \in \mathscr{GR}_{V_{\F},0,\underline{a},\underline{b}}(\F')$ 
can be written as 
$\Biggl(
 \begin{pmatrix}
 1 & v_i \\ 0 & 1
 \end{pmatrix} 
 \Biggr)_i \cdot
 \mathfrak{M} _{\F'}$ 
for some $(v_i )_{1 \leq i \leq n} \in \F'((u))^n$. 
With the basis induced from $\mathfrak{M} _{\F'}$, 
we have 
\begin{align*}
 \mathfrak{M}' _{\F'} \sim & 
 \Biggl( 
 \begin{pmatrix}
 1 & \phi(v_i) \\ 0 & 1
 \end{pmatrix}
 \begin{pmatrix}
 \alpha_i u^{a_i} & w_i \\ 0 & \beta_i u^{b_i}
 \end{pmatrix}
 \begin{pmatrix}
 1 & -v_{i+1} \\ 0 & 1
 \end{pmatrix}
 \Biggr)_i \\
 &=
 \Biggl(
 \begin{pmatrix}
 \alpha_i u^{a_i} & 
 w_i -\alpha_i u^{a_i} v_{i+1}+\beta_i u^{b_i} \phi(v_i) 
 \\ 0 & \beta_i u^{b_i}
 \end{pmatrix}
 \Biggr)_i . 
\end{align*}
We are going to examine the condition 
for $(v_i )_{1 \leq i \leq n} \in \F'((u))^n$ 
to give a point of 
$\mathscr{GR}_{V_{\F},0,\underline{a},\underline{b}}(\F')$
as 
$\Biggl(
 \begin{pmatrix}
 1 & v_i \\ 0 & 1
 \end{pmatrix} 
 \Biggr)_i \cdot
 \mathfrak{M} _{\F'}$. 
Extending the field $\F$, 
we may assume that 
$\mathscr{GR}_{V_{\F},0,\underline{a},\underline{b}}(\F) =\emptyset$ 
if and only if 
$\mathscr{GR}_{V_{\F},0,\underline{a},\underline{b}}(\F') =\emptyset$ 
for each $(\underline{a},\underline{b}) \in I$ and 
any finite extension $\F'$ of $\F$. 

For $(v_i )_{1 \leq i \leq n} \in \F'((u))^n$, 
we have 
$\mathfrak{M}' _{\F'} =
 \Biggl(
 \begin{pmatrix}
 1 & v_i \\ 0 & 1
 \end{pmatrix} 
 \Biggr)_i \cdot
 \mathfrak{M} _{\F'} 
 \in 
 \mathscr{GR}_{V_{\F},0,\underline{a},\underline{b}}(\F')$ 
if and only if 
\begin{align*}
 v_u \bigl( w_i -\alpha_i & u^{a_i} v_{i+1} 
 +\beta_i u^{b_i} \phi(v_i )\bigl) \geq 0
 \textrm{ and }\\& 
 v_u (\alpha_i u^{a_i} ) +
 v_u (\beta_i u^{b_i} ) -
 v_u \bigl( w_i -\alpha_i u^{a_i} v_{i+1} 
 +\beta_i u^{b_i} \phi(v_i )\bigl) \leq e
 \textrm{ for all } i, 
\end{align*}
by the condition 
$u^e \mathfrak{M}'_{\F'} \subset 
 (1\otimes \phi ) \bigl( \phi ^* (\mathfrak{M}'_{\F'}) \bigr)
 \subset \mathfrak{M}'_{\F'}$. 
This is further equivalent to 
\[
 v_u \bigl(\alpha_i u^{a_i} v_{i+1} 
 -\beta_i u^{b_i} \phi(v_i )\bigl) \geq 
 \max \{0,a_i +b_i -e \}, 
\]
because $v_u (w_i) \geq \max \{0,a_i +b_i -e \}$. 
We put $r_i =-v_u (v_i )$, 
and note that 
\begin{align*}
 v_u (\alpha_{i-1} u^{a_{i-1}} v_i ) \geq 
 \max \{0,a_{i-1} +b_{i-1} -e \} &
 \Leftrightarrow 
 r_i \leq \min \{a_{i-1} ,e-b_{i-1} \}, \\
  v_u \bigl( \beta_i u^{b_i} \phi(v_i )\bigl) \geq 
 \max \{0,a_i +b_i -e \} &
 \Leftrightarrow 
 r_i \leq \min \biggl\{ 
 \frac{e-a_i}{p}, \frac{b_i}{p} \biggr\}. 
\end{align*}

We define an $\F'$-vector space 
$\widetilde{N}_{\underline{a},\underline{b},\F'}$ by 
\begin{align*}
 \widetilde{N}_{\underline{a},\underline{b},\F'} &=
 \bigl\{
 (v_1 ,\ldots ,v_n ) \in \F'((u))^n \bigm| \\
 & \hspace*{7em}
 v_u \bigl(\alpha_i u^{a_i} v_{i+1} 
 -\beta_i u^{b_i} \phi(v_i )\bigl) \geq 
 \max \{0,a_i +b_i -e \} 
 \textrm{ for all } i 
 \bigr\}. 
\end{align*}
We note that 
$\widetilde{N}_{\underline{a},\underline{b},\F'}
 \supset \F'[[u]]^n$, 
and put 
$N_{\underline{a},\underline{b},\F'}
 =\widetilde{N}_{\underline{a},\underline{b},\F'} 
 \big/ \F'[[u]]^n$. 
Then we have a bijection 
$N_{\underline{a},\underline{b},\F'} \to 
 \mathscr{GR}_{V_{\F},0,\underline{a},\underline{b}}(\F')$ 
by Lemma \ref{equiv}. 
We put 
$d_{\underline{a},\underline{b}} =\dim_{\F'}
 N_{\underline{a},\underline{b},\F'}$, 
and note that 
$\dim_{\F'} N_{\underline{a},\underline{b},\F'}$ 
is independent of finite extensions $\F'$ of $\F$. 

We take a basis 
$(\mathbf{v}_j )_{1 \leq j \leq d_{\underline{a},\underline{b}} }$ of 
$N_{\underline{a},\underline{b},\F}$ over $\F$, 
where $\mathbf{v}_j =(v_{j,1} ,\ldots ,v_{j,n} ) \in \F((u))^n$. 
Then, by Proposition \ref{description}, an 
$(\F[[u]] \otimes_{\F} 
 \F [X_1 ,\ldots ,X_{d_{\underline{a},\underline{b}}} ])$-module 
\[
 \mathfrak{M}' _{\F [X_1 ,\ldots ,X_{d_{\underline{a},\underline{b}}} ]}= 
 \Biggl(
 \begin{pmatrix}
 1 & \sum_j v_{j,i} X_j \\ 0 & 1
 \end{pmatrix} 
 \Biggr)_i \cdot
 (\mathfrak{M} _{\F} \otimes_{\F} 
 \F [X_1 ,\ldots ,X_{d_{\underline{a},\underline{b}}} ] )
\]
gives a morphism 
$f_{\underline{a},\underline{b}} : 
 \mathbb{A}^{d_{\underline{a},\underline{b}}} _{\F}
 \to \mathscr{GR}_{V_{\F},0}$ 
such that $f_{\underline{a},\underline{b}} (\F')$ is injective and 
the image of $f_{\underline{a},\underline{b}} (\F')$ is 
$\mathscr{GR}_{V_{\F},0,\underline{a},\underline{b}}(\F')$. 
Then we have (1) and 
\[ 
 d_{V_{\F}} =\max_{(\underline{a},\underline{b}) \in I,\ 
 \mathscr{GR}_{V_{\F},0,\underline{a},\underline{b}}(\F) \neq \emptyset } 
 \{ d_{\underline{a},\underline{b}} \}. 
\] 

 Before going into a proof of (2), 
we will examine 
$d_{\underline{a},\underline{b}}$ 
to evaluate $d_{V_{\F}}$. 
We put
\begin{align*}
 S_{\underline{a},\underline{b},i} &=\Biggl\{
 (0,\ldots ,0,v_i ,0, \ldots ,0) \in \F((u))^n 
 \Biggm|  v_i =u^{-r_i}, \\
 & \hspace*{14em}
 1 \leq r_i \leq 
 \min \biggl\{
 a_{i-1} ,e-b_{i-1} ,\frac{e-a_i}{p},\frac{b_i}{p} 
 \biggr\} \Biggr\} \\
\intertext{for $1 \leq i \leq n$,}
 S_{\underline{a},\underline{b},i,j} &=\Biggl\{
 (0,\ldots ,0,v_i ,v_{i+1},\ldots ,v_{i+j} ,0,\ldots ,0) \in \F((u))^n 
 \Biggm| v_i =u^{-r_i},\ \\ 
 & \hspace*{2.8em}
 1\leq r_i \leq \min \{a_{i-1} ,e-b_{i-1} \},\  
 \alpha_{i+l} u^{a_{i+l}} v_{i+l+1} = 
 \beta_{i+l} u^{b_{i+l}} \phi(v_{i+l}) \\ 
 & \hspace*{6.8em}
 \textrm{and }
 -\!v_u (v_{i+l+1} )> 
 \min \{a_{i+l} ,e-b_{i+l} \} 
 \textrm{ for } 
 0 \leq l \leq j-1,\ \\
 & \hspace*{16.5em}
 -v_u (v_{i+j}) \leq 
 \min \biggl\{
 \frac{e-a_{i+j}}{p},\frac{b_{i+j}}{p} 
 \biggr\} \Biggr\}\\ 
\intertext{for $1\leq i \leq n$ and $1 \leq j \leq n-1$, and}
 S_{\underline{a},\underline{b}} &=\Bigl\{
 (v_1,\ldots ,v_n ) \in \F((u))^n 
 \Bigm| 
 \alpha_i u^{a_i} v_{i+1} = 
 \beta_i u^{b_i } \phi(v_i ),\ 
 v_1 =u^{v_u (v_1 )} \\ 
 & \hspace*{12.4em}
 \textrm{and }
 -\!v_u (v_{i+1} ) > 
 \min \{a_i ,e-b_i \} 
 \textrm{ for all } i 
 \Bigr\}.
\end{align*}
In the above definitions, $v_i$ is on the $i$-th component. 
Clearly, all elements of 
$\bigcup_{i} S_{\underline{a},\underline{b},i} 
 \cup \bigcup_{i,j} S_{\underline{a},\underline{b},i,j} 
 \cup S_{\underline{a},\underline{b}}$ 
are in 
$\widetilde{N}_{\underline{a},\underline{b},\F}$. 

\begin{lem}\label{basis}
The image of 
$\bigcup_{i} S_{\underline{a},\underline{b},i} 
 \cup \bigcup_{i,j} S_{\underline{a},\underline{b},i,j} 
 \cup S_{\underline{a},\underline{b}}$ 
in $N_{\underline{a},\underline{b},\F}$ 
forms an $\F$-basis of $N_{\underline{a},\underline{b},\F}$. 
\end{lem}
\begin{proof}
It is clear that the image of 
$\bigcup_{i} S_{\underline{a},\underline{b},i} 
 \cup \bigcup_{i,j} S_{\underline{a},\underline{b},i,j} 
 \cup S_{\underline{a},\underline{b}}$ 
in $N_{\underline{a},\underline{b},\F}$ 
are linearly independent over $\F$. 
So it suffices to show that 
$\bigcup_{i} S_{\underline{a},\underline{b},i} 
 \cup \bigcup_{i,j} S_{\underline{a},\underline{b},i,j} 
 \cup S_{\underline{a},\underline{b}}$ and 
$\F [[u]]^n$ generates 
$\widetilde{N}_{\underline{a},\underline{b},\F}$. 
We take 
$(v_1 ,\ldots ,v_n ) \in \widetilde{N}_{\underline{a},\underline{b},\F}$. 
We want to write 
$(v_1 ,\ldots ,v_n )$ as 
a linear combination of elements of 
$\bigcup_{i} S_{\underline{a},\underline{b},i} 
 \cup \bigcup_{i,j} S_{\underline{a},\underline{b},i,j} 
 \cup S_{\underline{a},\underline{b}}$ and 
$\F [[u]]^n$. 

First, we consider the case where 
there exsits an index $i_0$ such that 
$-v_u (v_{i_0} ) >\min 
 \{ a_{i_0 -1} ,e-b_{i_0 -1} ,(e-a_{i_0} )/p ,b_{i_0} /p \}$. 
Then there are following two cases: 
\begin{enumerate}
\item[(i)]
There are $1 \leq i_1 \leq n$ and $1 \leq j_1 \leq n-1$ such that \\ 
$i_0 \in [i_1 ,i_1 +j_1 ]$, 
$1 \leq -v_u (v_{i_1} ) \leq \min \{a_{i_1 -1} ,e-b_{i_1 -1} \}$, \\
$a_{i_1 +l} +v_u (v_{i_1 +l+1} ) = b_{i_1 +l} +p v_u (v_{i_1 +l})$\\ 
and $-v_u (v_{i_1 +l+1} )> \min \{a_{i+l} ,e-b_{i+l} \}$ 
for $0 \leq l \leq j_1 -1$\\ 
and $-v_u (v_{i_1 +j_1}) \leq \min 
 \{(e-a_{i_1 +j_1} )/p,(b_{i_1 +j_1 } )/p \}$. 
\item[(ii)]
$a_i +v_u (v_{i+1} )=b_i +pv_u (v_i )$ and 
$-v_u (v_{i+1} ) > \min \{a_i ,e-b_i \}$ for all $i$.
\end{enumerate}
In the case (i), 
we can subtract a linear multiple of an element of 
$S_{\underline{a},\underline{b},i_1 ,j_1 }$ 
from 
$(v_1 ,\ldots ,v_n )$ so that 
the $u$-valuations of the $i$-th component increase 
for all $i \in [i_1 ,i_1 +j_1 ]$. 
In the case (ii), 
we can subtract a linear multiple of an element of 
$S_{\underline{a},\underline{b}}$ 
from 
$(v_1 ,\ldots ,v_n )$ so that 
the $u$-valuations of the $i$-th component increase 
for all $i$. 

Repeating such subtractions, 
we may assume that 
$-v_u (v_{i} ) \leq \min 
 \{ a_{i-1} ,e-b_{i-1} ,(e-a_i )/p ,b_i /p \}$ for all $i$. 
Then we can write 
$(v_1 ,\ldots ,v_n )$ as 
a linear combination of elements of 
$\bigcup_{i} S_{\underline{a},\underline{b},i}$ and 
$\F [[u]]^n$. 
\end{proof}

By Lemma \ref{basis}, we have 
$d_{\underline{a},\underline{b}} =
 \sum_i |S_{\underline{a},\underline{b},i}| + 
 \sum_{i,j} |S_{\underline{a},\underline{b},i,j}| 
 +|S_{\underline{a},\underline{b}}|$. 
We note that 
$0 \leq |S_{\underline{a},\underline{b}}| \leq 1$ 
by the definition, and put 
$d' _{\underline{a},\underline{b}} =
 \sum_i |S_{\underline{a},\underline{b},i}| + 
 \sum_{i,j} |S_{\underline{a},\underline{b},i,j}|$. 

We put 
\[
 T_{\underline{a},\underline{b},i} =\Biggl\{
 m \in \mathbb{Z} \Biggm| 
 \min \{a_{i-1} ,e-b_{i-1} \} < pm +a_{i-1} -b_{i-1} 
 \leq \min \biggl\{
 \frac{e-a_i }{p}, \frac{b_i }{p} 
 \biggr\} \Biggr\}, 
\]
and consider the map
\[
 \bigcup_{i+j=h} S_{\underline{a},\underline{b},i,j} 
 \to T_{\underline{a},\underline{b},h} ;\ 
 (v_{i'})_{1 \leq i' \leq n} \mapsto -v_u (v_{h-1}). 
\]
We can easily check that this map is injective. 
So we have 
$\sum_{i+j=h} |S_{\underline{a},\underline{b},i,j}| 
 \leq |T_{\underline{a},\underline{b},h}|$ 
and 
$d' _{\underline{a},\underline{b}} \leq 
 \sum_{1\leq i \leq n} \bigl( 
 |S_{\underline{a},\underline{b},i}|+
 |T_{\underline{a},\underline{b},i}| \bigr)$. 

We take 
$(\underline{a}',\underline{b}') \in I$  
such that 
$\sum_{1\leq i \leq n} \bigl( 
 |S_{\underline{a}',\underline{b}',i}|
 +|T_{\underline{a}',\underline{b}',i}| \bigr)$ 
is the maximum. 

\begin{lem}\label{boundT}
$|T_{\underline{a}',\underline{b}',i}| \leq 1$ for all $i$. 
\end{lem}
\begin{proof}
We assume there is an index $i_0$ 
such that $|T_{\underline{a}',\underline{b}',i_0}| \geq 2$. 
We note that 
\begin{equation}
 \min \{a_{i_0 -1} ',e-b_{i_0 -1} '\} 
 +p+1 \leq 
 \min \biggl\{
 \frac{e-a_{i_0} '}{p}, \frac{b_{i_0} '}{p} 
 \biggr\} \tag{$*$} \label{eq:a',b'}
\end{equation}
by $|T_{\underline{a}',\underline{b}',i_0}| \geq 2$. 
We are going to show that we can replace 
$a_{i_0 -1} ',b_{i_0 -1} '$ so that 
$\sum_{1\leq i \leq n} 
 \bigl( |S_{\underline{a}',\underline{b}',i}|+
 |T_{\underline{a}',\underline{b}',i}| \bigr)$ 
increases. 
This contradicts the maximality of 
$\sum_{1\leq i \leq n} \bigl( 
 |S_{\underline{a}',\underline{b}',i}|
 +|T_{\underline{a}',\underline{b}',i}| \bigr)$. 
We divide the problem into three cases. 

Firstly, if 
$a_{i_0 -1} '+2 \leq e-b_{i_0 -1} '$, 
we replace $a_{i_0 -1} '$ by
$a_{i_0 -1} ' +p$, 
and note that 
$a_{i_0 -1} '+p \leq e$ by (\ref{eq:a',b'}). 
Then there is no change except for 
$S_{\underline{a}',\underline{b}',i_0 -1}$, 
$S_{\underline{a}',\underline{b}',i_0}$, 
$T_{\underline{a}',\underline{b}',i_0 -1}$ and 
$T_{\underline{a}',\underline{b}',i_0}$. 
We can see that 
$|S_{\underline{a}',\underline{b}',i_0}|$ increases 
by at least $2$. 
The condition that 
there exists $m \in \mathbb{Z}$ such that 
\[
 \min \{a_{i_0 -1} ',e-b_{i_0 -1} '\} 
 < pm+a_{i_0 -1} '-b_{i_0 -1} ' \leq 
 \min \{a_{i_0 -1} '+p,e-b_{i_0 -1} '\}, 
\]
is equivalent to the condition that 
there exists $m \in \mathbb{Z}$ such that
\[
 \min \biggl\{
 \frac{e-a_{i_0 -1} '}{p} ,
 \frac{b_{i_0 -1} '}{p} 
 \biggr\} < m \leq 
 \min \biggl\{
 \frac{e-a_{i_0 -1} '}{p} ,
 \frac{b_{i_0 -1} '}{p} +1 
 \biggr\}, 
\]
and further equivalent to the condition 
that there does not exists $m \in \mathbb{Z}$ such that
\[
 \min \biggl\{
 \frac{e-a_{i_0 -1} '}{p} -1,
 \frac{b_{i_0 -1} '}{p} 
 \biggr\} < m \leq 
 \min \biggl\{
 \frac{e-a_{i_0 -1} '}{p} ,
 \frac{b_{i_0 -1} '}{p} 
 \biggr\}. 
\]
If the above condition is satisfied, 
then 
$|S_{\underline{a}',\underline{b}',i_0 -1}|$, 
$|T_{\underline{a}',\underline{b}',i_0 -1}|$ 
do not change 
and $|T_{\underline{a}',\underline{b}',i_0}|$ 
decreases by $1$. 
Otherwise, 
$|S_{\underline{a}',\underline{b}',i_0 -1}| +
 |T_{\underline{a}',\underline{b}',i_0 -1}|$ 
decreases by at most $1$ 
and $|T_{\underline{a}',\underline{b}',i_0}|$ does not change. 
In both cases, we have that 
$\sum_{1\leq i \leq n} 
 \bigl( |S_{\underline{a}',\underline{b}',i}|+
 |T_{\underline{a}',\underline{b}',i}| \bigr)$ 
increases by at least $1$. 

Secondly, if 
$a_{i_0 -1} '\geq e-b_{i_0 -1} '+2$, 
we replace $b_{i_0 -1} '$ by
$b_{i_0 -1} '-p$. 
Then, by the same arguments, 
we have that 
$\sum_{1\leq i \leq n} 
 \bigl( |S_{\underline{a}',\underline{b}',i}|+
 |T_{\underline{a}',\underline{b}',i}| \bigr)$ 
increases by at least $1$. 

In the remaining case, that is the case where 
$a_{i_0 -1} ' -1 \leq e-b_{i_0 -1} ' \leq a_{i_0 -1} '+1$, 
we replace 
$a_{i_0 -1} '$, $b_{i_0 -1} '$ by 
$a_{i_0 -1} ' +p$, $b_{i_0 -1} ' -p$ respectively, 
and note that $a_{i_0 -1} ' +p \leq e$ and 
$b_{i_0 -1} ' -p \geq 0$ by (\ref{eq:a',b'}). 
Then there is no change except for 
$S_{\underline{a}',\underline{b}',i_0 -1}$,
$S_{\underline{a}',\underline{b}',i_0}$,
$T_{\underline{a}',\underline{b}',i_0 -1}$ and 
$T_{\underline{a}',\underline{b}',i_0}$. 
We can see that 
$|S_{\underline{a}',\underline{b}',i_0 -1}| 
 +|T_{\underline{a}',\underline{b}',i_0 -1}|$ 
decreases by at most $1$, 
$|S_{\underline{a}',\underline{b}',i_0}|$ 
increases by $p$ 
and $|T_{\underline{a}',\underline{b}',i_0}|$ 
decreases by $1$. 
Hence 
$\sum_{1\leq i \leq n} 
 \bigl( |S_{\underline{a}',\underline{b}',i}|
 +|T_{\underline{a}',\underline{b}',i}| \bigr)$ 
increases by at least $p-2>0$. 

Thus we have proved that 
$|T_{\underline{a}',\underline{b}',i}| \leq 1$ for all $i$. 
\end{proof}

\begin{lem}\label{AB}
For all $i$, we have the followings: 

\begin{enumerate}
 \item[$(A_i)$] 
If 
$|S_{\underline{a}',\underline{b}',i}|+
 |T_{\underline{a}',\underline{b}',i}| = e_0 +l$ 
for $l \geq 1$,\\ 
then 
$|S_{\underline{a}',\underline{b}',i+1} |+
 |T_{\underline{a}',\underline{b}',i+1} | 
 \leq e_0 +e_1 -pl+1$. 
 \item[$(B_i)$] 
If 
$|S_{\underline{a}',\underline{b}',i}|+
 |T_{\underline{a}',\underline{b}',i}| = e_0 +1$\\ 
and  
$|S_{\underline{a}',\underline{b}',i+1} |+
 |T_{\underline{a}',\underline{b}',i+1} | 
 = e_0 +e_1 -p+1$,\\ 
then 
$|S_{\underline{a}',\underline{b}',i+2} |+
 |T_{\underline{a}',\underline{b}',i+2} | 
 \leq e_0 -(p-1)e_1 +1$. 
\end{enumerate}

\end{lem}

\begin{proof}
By the definition of 
$T_{\underline{a},\underline{b},i}$, we have 
\[
 |T_{\underline{a},\underline{b},i}| \leq 
 \max \Biggl\{ 
 \min \biggl\{
 \frac{e-a_i}{p}, \frac{b_i}{p} 
 \biggr\} - 
 \min \{a_{i-1} ,e-b_{i-1}\}, 
 0
 \Biggr\}. 
\]
Combining this with the definition of 
$S_{\underline{a},\underline{b},i}$, we get 
\begin{equation}
 |S_{\underline{a},\underline{b},i}| +
 |T_{\underline{a},\underline{b},i}| \leq 
 \min \Biggl\{ 
 \biggl[ \frac{e-a_i}{p} \biggr] , 
 \biggl[ \frac{b_i}{p} \biggr] 
 \Biggr\}, \tag{$\star$} \label{eq:S+T}
\end{equation}
and equality happens if and only if 
in the following two cases: 
\begin{itemize}
 \item
 $\min \biggl\{
 \Bigl[ \frac{e-a_i}{p} \Bigr], 
 \Bigl[ \frac{b_i}{p} \Bigr] 
 \biggr\} - 
 \min \{a_{i-1} ,e-b_{i-1}\} \leq 0$. 
 \item 
 $\min \biggl\{
 \Bigl[ \frac{e-a_i}{p} \Bigr], 
 \Bigl[ \frac{b_i}{p} \Bigr] 
 \biggr\} - 
 \min \{a_{i-1} ,e-b_{i-1}\} =1$ \\
 and 
 $p \mid \bigl( \min \{e-a_{i-1} ,b_{i-1} \} +1 \bigr)$.
\end{itemize}

We assume 
$|S_{\underline{a}',\underline{b}',i_1}|+
 |T_{\underline{a}',\underline{b}',i_1}| = e_0 +l$ 
for some $i_1$ and $l \geq 1$. 
Then we have 
$p(e_0 +l) \leq \min \{e-a_{i_1} ' ,b_{i_1} '\}$ 
by (\ref{eq:S+T}). 
By this inequality, we have 
\begin{align*}
 |S_{\underline{a}',\underline{b}',i_1 +1}| 
 &\leq \min \{a_{i_1} ' ,e-b_{i_1} ' \} 
 \leq \max \{a_{i_1} ' ,e-b_{i_1} ' \}\\
 &=e- \min \{e-a_{i_1} ' ,b_{i_1} '\}
 \leq e-p(e_0 +l) = e_0 +e_1 -pl. 
\end{align*}
Combining this with 
$|T_{\underline{a}',\underline{b}',i_1 +1}| \leq 1$, 
we get 
\[
 |S_{\underline{a}',\underline{b}',i_1 +1}|+
 |T_{\underline{a}',\underline{b}',i_1 +1}| \leq 
 e_0 +e_1 -pl+1. 
\] 
This shows $(A_i)$ for all $i$. 

Further, we examine the case 
where equality holds in the above inequality, 
assuming $l=1$. 
In this case, we have that 
$\min \{a_{i_1} ' ,e-b_{i_1} ' \} =e_0 +e_1 -p$, 
$\min \{e-a_{i_1} ',b_{i_1} '\}=p(e_0 +1)$ 
and 
$|T_{\underline{a}',\underline{b}',i_1 +1}|=1$. 
Let $m$ be the unique element of 
$T_{\underline{a}',\underline{b}',i_1 +1}$. 
Then, by the definition of 
$T_{\underline{a}',\underline{b}',i_1 +1}$, 
we have 
\[ 
 \min \biggl\{
 \frac{e-a_{i_1 +1} '}{p}, \frac{b_{i_1 +1} '}{p} 
 \biggr\} - 
 \min \{a_{i_1} ',e-b_{i_1} '\} 
 \geq 
 pm - \min \{e-a_{i_1} ',b_{i_1} '\} 
 \geq p, 
\] 
because 
$\min \{e-a_{i_1} ',b_{i_1} '\}=p(e_0 +1)$ 
and 
$pm - \min \{e-a_{i_1} ',b_{i_1} '\} >0$. 
Combining this with 
$\min \{a_{i_1} ',e-b_{i_1} '\}=e_0 +e_1 -p$, 
we get 
$p(e_0 +e_1) \leq \min 
 \{e-a_{i_1 +1} ',b_{i_1 +1} '\}$. 
By the previous argument, 
we have 
\[
 |S_{\underline{a}',\underline{b}',i_1 +2}|+
 |T_{\underline{a}',\underline{b}',i_1 +2}| \leq 
 e_0 -(p-1)e_1 +1. 
\] 
Thus we have proved $(B_i)$ for all $i$. 
\end{proof}

We are going to show (2). 
Firstly, we treat (a). 
We note that 
$e_0 +e_1 -pl+1 \leq e_0 -p(l-1) -2$
in the case where $0 \leq e_1 \leq p-3$, 
and that 
$e_0 +e_1 -pl+1 \leq e_0 -p(l-1) -1$ and 
$e_0 -(p-1)e_1 +1 \leq e_0 -1$ 
in the case where $e_1 = p-2$. 
Then $(A_i)$ and $(B_i)$ for all $i$ 
implies that 
 $\sum_{1\leq i \leq n} 
 \bigl( |S_{\underline{a}',\underline{b}',i}|+
 |T_{\underline{a}',\underline{b}',i}| \bigr) 
 \leq ne_0$. 
It further implies that 
\[
 d' _{\underline{a},\underline{b}} \leq 
 \sum_{1\leq i \leq n} 
 \bigl( |S_{\underline{a},\underline{b},i}|+
 |T_{\underline{a},\underline{b},i}| 
 \bigr) \leq ne_0 
\] 
for all 
$(\underline{a},\underline{b}) \in I$, and 
that $d' _{\underline{a},\underline{b}} = ne_0$ 
only if 
$|S_{\underline{a},\underline{b},i}|+
 |T_{\underline{a},\underline{b},i}| = e_0$
for all $i$. 
To prove $d_{\underline{a},\underline{b}} \leq ne_0$, 
it suffice to show that 
$d' _{\underline{a},\underline{b}} = ne_0$ implies 
$S_{\underline{a},\underline{b}} =\emptyset$, 
because 
$|S_{\underline{a},\underline{b}}| \leq 1$ 
for all $(\underline{a},\underline{b}) \in I$. 

We assume that 
$d' _{\underline{a},\underline{b}} = ne_0$ and 
$S_{\underline{a},\underline{b}} \neq \emptyset$. 
By the maximality of 
$\sum_{1\leq i \leq n} 
 \bigl( |S_{\underline{a},\underline{b},i}|+
 |T_{\underline{a},\underline{b},i}| \bigr)$, 
we have $|T_{\underline{a},\underline{b},i}| \leq 1$ 
for all $i$. 
Let 
$(v_{0,i})_{1\leq i \leq n}$ be 
the unique element of 
$S_{\underline{a},\underline{b}}$, 
and we put $r_{0,i} =-v_u (v_{0,i})$. 
Then we have 
\[
 a_i -r_{0,i+1} =b_i -pr_{0,i} < \max\{0,a_i +b_i -e\} 
\]
for all $i$, by the definition of 
$S_{\underline{a},\underline{b}}$. 
By (\ref{eq:S+T}) and 
$e_0 -1 \leq |S_{\underline{a},\underline{b},i}|$ 
for all $i$, 
we have 
\[
 e_0 -1 \leq a_i \leq e_0 + e_1 ,\ 
 pe_0 \leq b_i \leq pe_0 +e_1 +1 
\]
for all $i$. 
Take an index $i_2$ such that 
$r_{0,i_2}$ is the maximum. 
Then we have 
\begin{align*}
 (p-1)r_{0,i_2} \leq pr_{0,i_2} -r_{0,i_2 +1} 
 &=b_{i_2} -a_{i_2} \leq 
 (pe_0 +e_1 +1) -(e_0 -1)\\
 &=(p-1)e_0 +e_1 +2 \leq 
 (p-1)e_0 +p. 
\end{align*}
So we get $r_{0,i} \leq e_0 +1$ 
for all $i$. 

If $a_i + b_i -e \leq 0$, 
we have $r_{0,i} \geq e_0 +1$ 
by $b_i -pr_{0,i} <0$ and 
$pe_0 \leq b_i$. 
If $a_i + b_i -e > 0$, 
we have $r_{0,i} \geq e_0 +1$ 
by $b_i -pr_{0,i} < a_i +b_i -e$ and 
$a_i \leq e_0 +e_1$. 
So we have  
$r_{0,i} = e_0 +1$ 
for all $i$. 

By $a_i -r_{0,i+1} =b_i -pr_{0,i}$, 
we have 
$(p-1)(e_0 +1)=b_i -a_i$ for all $i$. 
By the range of $a_i$ and $b_i$, 
we have the following two possibilities for each $i$: 
\[
 (a_i ,b_i )= 
 (e_0 -1,pe_0 +p-2) 
 \textrm{ or }
 (e_0 ,pe_0 +p-1).
\]
In both cases, we have 
$|S_{\underline{a},\underline{b},i+1}| =e_0 -1$.

Now we must have equality in (\ref{eq:S+T}). 
So we must have 
$p \mid ( \min \{ e-a_{i-1} ,b_{i-1} \}+1)$, 
noting that 
$|T_{\underline{a},\underline{b},i}|=1$. 
This contradicts the possibilities of 
$a_{i-1} ,b_{i-1}$. 
Thus we have proved 
$d_{V_{\F}} \leq ne_0$. 
 
For $\underline{a} =(e_0)_{1 \leq i \leq n}$ and
$\underline{b} =(pe_0)_{1 \leq i \leq n}$, 
we have 
$d_{\underline{a},\underline{b}} \geq 
 \sum_{1 \leq i \leq n} |S_{\underline{a},\underline{b},i}| 
 =ne_0$. 
This shows that 
$d_{V_{\F}} =ne_0$, 
if 
\[
 M_{\F} \sim
 \Biggl( 
 \begin{pmatrix}
  u^{e_0} & 0 \\ 0 & u^{pe_0}
 \end{pmatrix}
 \Biggr)_i .
\] 

Secondly, we treat (b). 
In this case, 
we note that 
$e_0 +e_1 -pl+1 = e_0 -p(l-1)$ and 
$e_0 -(p-1)e_1 +1 \leq e_0 -3$. 
Then $(A_i)$ and $(B_i)$ for all $i$ 
implies $d' _{\underline{a},\underline{b}} \leq ne_0$, 
and further implies 
$d_{\underline{a},\underline{b}} \leq ne_0 +1$, 
because $|S_{\underline{a},\underline{b}}| \leq 1$. 
Thus we have proved 
$d_{V_{\F}} \leq ne_0 +1$. 

For $\underline{a} =(e_0)_{1 \leq i \leq n}$ and
$\underline{b} =(pe_0 +p-1)_{1 \leq i \leq n}$, 
we have 
$d_{\underline{a},\underline{b}} \geq 
 \sum_{1 \leq i \leq n} |S_{\underline{a},\underline{b},i}| 
 + |S_{\underline{a},\underline{b}}| 
 =ne_0 +1$, 
because 
$(u^{-(e_0 +1)})_{1 \leq i \leq n} \in 
 S_{\underline{a},\underline{b}}$. 
This shows that 
$d_{V_{\F}} =ne_0 +1$, 
if 
\[
 M_{\F} \sim
 \Biggl( 
 \begin{pmatrix}
  u^{e_0} & 0 \\ 0 & u^{pe_0 +p-1}
 \end{pmatrix}
 \Biggr)_i . 
\]

At last, we treat (c).  
In this case, 
we note that 
$e_0 +e_1 -pl+1 = e_0 -p(l-1)+1$ and 
$e_0 -(p-1)e_1 +1 \leq e_0 -5$. 
Then $(A_i)$ and $(B_i)$ for all $i$ 
implies 
$d' _{\underline{a},\underline{b}} \leq ne_0 +[n/2]$, 
and that 
$d' _{\underline{a},\underline{b}} = ne_0 +[n/2]$ 
only if 
$e_0 \leq 
 |S_{\underline{a},\underline{b},i}|+
 |T_{\underline{a},\underline{b},i}| \leq e_0 +1$ 
for all $i$. 

If $n=1$, then 
$d' _{\underline{a},\underline{b}} \leq e_0$ 
implies 
$d_{\underline{a},\underline{b}} \leq e_0 +1$, 
and the given example for 
$d_{V_{\F}} =e_0 +1$ 
is the same as in (b). 
So we may assume $n \geq 2$ in the following. 

To prove 
$d_{\underline{a},\underline{b}} \leq ne_0 +[n/2]$, 
it suffices to show that 
$d' _{\underline{a},\underline{b}} = ne_0 +[n/2]$ implies 
$S_{\underline{a},\underline{b}} =\emptyset$, 
because 
$|S_{\underline{a},\underline{b}}| \leq 1$ 
for all $(\underline{a},\underline{b}) \in I$. 

We assume that 
$d' _{\underline{a},\underline{b}} = ne_0 +[n/2]$ and 
$S_{\underline{a},\underline{b}} \neq \emptyset$. 
By the maximality of 
$\sum_{1\leq i \leq n} 
 \bigl( |S_{\underline{a},\underline{b},i}|+
 |T_{\underline{a},\underline{b},i}| \bigr)$, 
we have $|T_{\underline{a},\underline{b},i}| \leq 1$ 
for all $i$. 
Let 
$(v_{1,i})_{1\leq i \leq n}$ be 
the unique element of 
$S_{\underline{a},\underline{b}}$, 
and we put $r_{1,i} =-v_u (v_{1,i})$. 
Then we have 
\[
 a_i -r_{1,i+1} =b_i -pr_{1,i} < \max\{0,a_i +b_i -e\} 
\]
for all $i$ by the definition of 
$S_{\underline{a},\underline{b}}$. 
By (\ref{eq:S+T}) and 
$e_0 -1 \leq |S_{\underline{a},\underline{b},i}|$, 
we have 
\[
 e_0 -1 \leq a_i \leq e_0 +p,\ 
 pe_0 \leq b_i \leq pe_0 +p +1
\]
for all $i$. 
Take an index $i_3$ such that 
$r_{1,i_3}$ is the maximum. 
Then we have 
\begin{align*}
 (p-1)r_{1,i_3} &\leq pr_{1,i_3} -r_{1,i_3 +1} 
 =b_{i_3} -a_{i_3} \\
 &\leq 
 (pe_0 +p +1) -(e_0 -1)
 =(p-1)e_0 +p+2. 
\end{align*}
So we get $r_{1,i} \leq e_0 +2$ 
for all $i$.

If $a_i + b_i -e \leq 0$, 
we have $r_{1,i} \geq e_0 +1$ 
by $b_i -pr_{1,i} <0$ and 
$pe_0 \leq b_i$. 
If $a_i + b_i -e > 0$, 
we have $r_{1,i} \geq e_0 +1$ 
by $b_i -pr_{1,i} < a_i +b_i -e$ and 
$a_i \leq e_0 +p$. 
So we have  
$e_0 +1 \leq r_{1,i} \leq e_0 +2$ 
for all $i$. 

By $n \geq 2$, there is an index $i_4$ 
such that 
$|S_{\underline{a},\underline{b},i_4}| + 
 |T_{\underline{a},\underline{b},i_4}|=e_0 +1$. 
Then we have 
$e_0 +1 \leq \min \bigl\{
 (e-a_{i_4} )/p,b_{i_4} /p \bigr\}$ 
by (\ref{eq:S+T}). 
We are going to prove that 
if 
$e_0 +1 \leq 
 \min \bigl\{ (e-a_i )/p,b_i /p \bigr\}$, 
then 
$|S_{\underline{a},\underline{b},i+1}| + 
 |T_{\underline{a},\underline{b},i+1}| =e_0$ 
and 
$e_0 +1 \leq \min \bigl\{
 (e-a_{i+1} )/p,b_{i+1} /p \bigr\}$. 
If we have proved this claim, 
we have a contradiction 
by considering $i_4$. 

We assume that 
$e_0 +1 \leq \min \bigl\{
 (e-a_i )/p,b_i /p \bigr\}$. 
Then we have 
$e_0 -1 \leq a_i \leq e_0$, 
$pe_0 +p \leq b_i \leq pe_0 +p+1$ and 
$e_0 -1 \leq |S_{\underline{a},\underline{b},i+1}| \leq e_0$. 
If 
$|S_{\underline{a},\underline{b},i+1}|=e_0$, 
we have $a_i =e_0$ and $b_i =pe_0 +p$. 
However, this contradicts 
$pr_i -r_{i+1} =b_i -a_i$, 
because $pr_i -r_{i+1} \neq (p-1)e_0 +p$ 
by 
$e_0 +1 \leq r_i ,r_{i+1} \leq e_0 +2$. 
So we have 
$|S_{\underline{a},\underline{b},i+1}|=e_0 -1$ 
and 
$|T_{\underline{a},\underline{b},i+1}|=1$. 
Let $m$ be the unique element of 
$T_{\underline{a},\underline{b},i+1}$. 
By the definition of 
$T_{\underline{a},\underline{b},i+1}$, 
we have 
\[
 \min \biggl\{ 
 \frac{e-a_{i+1}}{p},\frac{b_{i+1}}{p} 
 \biggr\} - 
 \min \{a_i ,e-b_i \} \geq 
 pm - 
 \min \{e-a_i ,b_i \}
 \geq p-1 \geq 2, 
\]
because 
$pe_0 +p \leq \min \{e-a_i ,b_i \} 
 \leq pe_0 +p+1$ 
and 
$pm- \min \{e-a_i ,b_i \} >0$. 
This shows 
$e_0 +1 \leq \min \bigl\{
 (e-a_{i+1} )/p,b_{i+1} /p \bigr\}$. 
Thus we have proved that 
$d_{V_{\F}} \leq ne_0 +[n/2]$. 
 
For $\underline{a} =(e_{0,i})_{1 \leq i \leq n}$ and
$\underline{b} =\bigl( p(2e_0 +1 -e_{0,i})\bigr)_{1 \leq i \leq n}$, 
we have 
\[
 d_{\underline{a},\underline{b}} \geq 
 \sum_{1 \leq i \leq n} |S_{\underline{a},\underline{b},i}|  
 =ne_0 +[n/2],
\] 
where $e_{0,i}$ is defined in the statement of 
Proposition \ref{reducible}(2)(c). 
This shows that 
$d_{V_{\F}} =ne_0 +[n/2]$, 
if 
\[ 
 M_{\F} \sim
 \Biggl( 
 \begin{pmatrix}
  u^{e_{0,i}} & 0 \\ 0 & u^{p(2e_0 +1 -e_{0,i} )} 
 \end{pmatrix}
 \Biggr)_i .
\]
\end{proof}

\section{The case where $V_{\F}$ is absolutely irreducible}

In this section, 
we give the maximum of 
the dimensions of the moduli spaces 
in the case where 
$V_{\F}$ is absolutely irreducible. 
In the proof of the following Proposition, 
three Lemmas appear. 

\begin{prop}\label{irreducible}
We assume $V_{\F}$ is absolutely irreducible, 
and write $e=(p+1)e_0 +e_1$ for $e_0 \in \mathbb{Z}$ 
and $0 \leq e_1 \leq p$. 
Then the followings are true. 
\begin{enumerate}
\item 
There are $m_i \in \mathbb{Z}$ 
for $0 \leq i \leq d_{V_{\F}}$ 
such that $m_{d_{V_{\F}}} >0$ and 
\[
 |\mathscr{GR}_{V_{\F},0} (\F')| =
 \sum_{i=0} ^{d_{V_{\F}}} 
 m_i |\F'|^i 
\]
for all sufficiently large 
extensions $\F'$ of $\F$. 
\item 
\begin{enumerate}
\item
In the case $e_1 =0$, 
we have $d_{V_{\F}} \leq ne_0 -1$. 
In this case, if 
\[
 \hspace*{4em} 
 M_{\F} \sim
 \Biggl( 
 \begin{pmatrix}
  0 & 1 \\ u^{(p+1)e_0 -1} & 0 
 \end{pmatrix}, 
 \begin{pmatrix}
  u^{e_0} & 0 \\ 0 & u^{pe_0}
 \end{pmatrix}, \ldots , 
 \begin{pmatrix}
  u^{e_0} & 0 \\ 0 & u^{pe_0}
 \end{pmatrix}
 \Biggr) , 
\]
then $d_{V_{\F}} = ne_0 -1$. 
\item 
In the case $1 \leq e_1 \leq p-1$, 
we have $d_{V_{\F}} \leq ne_0$. 
In this case, if 
\[
 \hspace*{4em} 
 M_{\F} \sim
 \Biggl( 
 \begin{pmatrix}
  0 & 1 \\ u^{(p+1)e_0 +1} & 0 
 \end{pmatrix}, 
 \begin{pmatrix}
  u^{e_0} & 0 \\ 0 & u^{pe_0}
 \end{pmatrix}, \ldots , 
 \begin{pmatrix}
  u^{e_0} & 0 \\ 0 & u^{pe_0}
 \end{pmatrix}
 \Biggr) , 
\]
we have $d_{V_{\F}} = ne_0$. 
\item 
In the case $e_1 = p$, 
we have 
$d_{V_{\F}} \leq ne_0 +[n/2]$. 
In this case, if 
\[
 \hspace*{2.2em}
 M_{\F} \sim
 \Biggl( 
 \begin{pmatrix}
  0 & 1 \\ u^{(p+1)e_0 +1} & 0 
 \end{pmatrix},  
 \begin{pmatrix}
  u^{2e_0 +1 -e_{0,i} } & 0 \\ 0 & u^{pe_{0,i}} 
 \end{pmatrix}_{2 \leq i \leq n} 
 \Biggr) , 
\] 
then $d_{V_{\F}} = ne_0 +[n/2]$. 
Here, $e_{0,i} =e_0$ 
if $i$ is odd, 
and $e_{0,i} =e_0 +1$ 
if $i$ is even. 
\end{enumerate}
\end{enumerate}
\end{prop}

\begin{proof}
Extending the field $\F$, 
we may assume that  
\[
 M_{\F} \sim 
\Biggl(
\begin{pmatrix}
 0 & \alpha _1 \\ \alpha _1 u^m & 0
\end{pmatrix}
,
\begin{pmatrix}
 \alpha _2 & 0 \\ 0 & \alpha _2
\end{pmatrix}
,
\ldots
,
\begin{pmatrix}
 \alpha _n & 0 \\ 0 & \alpha _n
\end{pmatrix}
\Biggr)
\]
for some $\alpha _i \in \F^{\times}$ and a positive integer $m$ 
such that $(q+1) \nmid m$, 
by Lemma \ref{stracture}. 
Let $\mathfrak{M} _{0,\F}$ be the lattice of 
$M_{\F}$ generated by the basis giving 
the above matrix expression. 

For any finite extension $\F'$ of $\F$, 
we put 
$\mathfrak{M} _{0,\F'} =
\mathfrak{M} _{0,\F} \otimes _{\F} \F'$ 
and 
$M_{\F'} =M_{\F} \otimes _{\F} \F'$. 
By the Iwasawa decomposition, 
any sublattice of $M_{\F'}$ can be written as 
$\Biggl(
 \begin{pmatrix}
  u^{s_i} & v_i ' \\ 0 & u^{t_i} 
 \end{pmatrix} 
 \Biggr)_i \cdot 
 \mathfrak{M} _{0,\F'}$
for $s_i ,t_i \in \mathbb{Z}$ and 
$v_i ' \in \F'((u))$. 

We put 
\begin{align*}
 \mathscr{GR}_{V_{\F},0,\underline{a},\underline{b}}(\F') &=
 \Biggl\{
 \Biggl( 
 \begin{pmatrix}
  u^{s_i} & v_i ' \\ 0 & u^{t_i} 
 \end{pmatrix} 
 \Biggr)_i \cdot 
 \mathfrak{M} _{0,\F'}
 \in
 \mathscr{GR}_{V_{\F},0}(\F')
 \Biggm| 
 s_i ,t_i \in \mathbb{Z}, \ 
 v_i ' \in \F'((u)), \\
 & \hspace*{7em}
 ps_1 -t_2 =a_1 ,\ m+pt_1 -s_2 =b_1, \\
 & \hspace*{9em}
 ps_j -s_{j+1} =a_j ,\ 
 pt_j -t_{j+1} =b_j 
 \textrm{ for } 2 \leq j \leq n 
 \Biggr\}
\end{align*}
for 
$(\underline{a},\underline{b})=
 \bigl( (a_i)_{1 \leq i \leq n} ,
 (b_i)_{1 \leq i \leq n}\bigr) 
 \in \mathbb{Z}^n \times \mathbb{Z}^n$. 
Then we have 
\[
 \mathscr{GR}_{V_{\F},0}(\F') =
 \bigcup_{(\underline{a},\underline{b})
 \in \mathbb{Z}^n \times \mathbb{Z}^n} 
 \mathscr{GR}_{V_{\F},0,\underline{a},\underline{b}}(\F') 
\]
and this is a disjoint union by Lemma \ref{equiv}.
Later, in Lemma \ref{range}, we will show that 
there are only finitely 
many $(\underline{a},\underline{b})$ 
such that 
$\mathscr{GR}_{V_{\F},0,\underline{a},\underline{b}}(\F')
 \neq \emptyset$. 

We take 
\[
 \Biggl(
 \begin{pmatrix}
  u^{s_i } & v_i ' \\ 0 & u^{t_i } 
 \end{pmatrix}
 \Biggr)_i \cdot 
 \mathfrak{M} _{0,\F'} 
 \in \mathscr{GR}_{V_{\F},0,\underline{a},\underline{b}}(\F'),
\] 
and put 
\[
 \mathfrak{M} _{\underline{a},\underline{b},\F'} = 
 \Biggl(
 \begin{pmatrix}
  u^{s_i } & 0 \\ 0 & u^{t_i } 
 \end{pmatrix}
 \Biggr)_i \cdot 
 \mathfrak{M} _{0,\F'}.
\] 
Then we have 
\[
 \mathfrak{M} _{\underline{a},\underline{b},\F'} \sim 
 \Biggl(
 \alpha_1
 \begin{pmatrix}
 0 & u^{a_1} \\ u^{b_1} & 0 
 \end{pmatrix}, 
 \alpha_2 
 \begin{pmatrix}
 u^{a_2} & 0 \\ 0 & u^{b_2} 
 \end{pmatrix}, \ldots , 
 \alpha_n
 \begin{pmatrix}
 u^{a_n} & 0 \\ 0 & u^{b_n} 
 \end{pmatrix}  
 \Biggr)
\] 
with respect to the basis induced from 
$\mathfrak{M} _{0,\F'}$.

Now, any 
$\mathfrak{M} _{\F'} \in 
 \mathscr{GR}_{V_{\F},0,\underline{a},\underline{b}}(\F')$ 
can be written as 
$\Biggl(
 \begin{pmatrix}
 1 & v_i \\ 0 & 1
 \end{pmatrix} 
 \Biggr)_i \cdot
 \mathfrak{M} _{\underline{a},\underline{b},\F'}$ 
for some $(v_i )_{1 \leq i \leq n} \in \F'((u))^n$, 
and we put $r_i =-v_u (v_i)$. 
We may assume $r_i \geq 0$, 
replacing 
$v_i$ so that $v_i \notin u\F'[[u]]$ 
without changing 
the $(k[[u]] \otimes _{\F_p} \F')$-module 
$\Biggl(
 \begin{pmatrix}
 1 & v_i \\ 0 & 1
 \end{pmatrix} 
 \Biggr)_i \cdot
 \mathfrak{M} _{\underline{a},\underline{b},\F'}$ 
by Lemma \ref{equiv}. 
Then we have 
\[ 
 \mathfrak{M}_{\F'} \sim 
 \Biggl( 
 \alpha_1
 \begin{pmatrix}
 \phi (v_1 )u^{b_1} & u^{a_1} -\phi(v_1 )v_2 u^{b_1} \\ 
 u^{b_1} & -v_2 u^{b_1} 
 \end{pmatrix}, 
 \alpha_i
 \begin{pmatrix}
 u^{a_i} & \phi(v_i )u^{b_i} -v_{i+1} u^{a_i} \\ 
 0 & u^{b_i}
 \end{pmatrix}_{2 \leq i \leq n} 
 \Biggr) 
\]
with respect to the induced basis, 
and 
\begin{align*}
 \begin{pmatrix}
 \phi (v_1 )u^{b_1} & u^{a_1} -\phi(v_1 )v_2 u^{b_1} \\ 
 u^{b_1} & -v_2 u^{b_1} 
 \end{pmatrix}
 &=
 \begin{pmatrix}
 \phi (v_1 )u^{b_1} & u^{a_1} \\ 
 u^{b_1} & 0 
 \end{pmatrix}
 \begin{pmatrix}
 1 & -v_2 \\ 
 0 & 1 
 \end{pmatrix}\\
 &=
 \begin{pmatrix}
 v_2 ^{-1} u^{a_1} & u^{a_1} -\phi(v_1 )v_2 u^{b_1} \\ 
 0 & -v_2 u^{b_1} 
 \end{pmatrix}
 \begin{pmatrix}
 1 & 0 \\ 
 -v_2 ^{-1} & 1 
 \end{pmatrix}.
\end{align*} 

Then the condition 
$u^e \mathfrak{M}_{\F'} \subset 
(1\otimes \phi ) \bigl( \phi ^* (\mathfrak{M}_{\F'}) \bigr)
\subset \mathfrak{M}_{\F'}$ 
is equivalent to 
\begin{equation}
\begin{split}
 0 \leq a_1 +r_2 \leq e,\ 
 0 \leq &b_1 -r_2 \leq e,\\ 
 &v_u \bigl( 
 u^{a_1} -\phi(v_1 )v_2 u^{b_1} \bigr)
 \geq \max \{ 0,a_1 +b_1 -e \}, 
\end{split} \tag{$C_1 $} \label{eq:C1}
\end{equation}
\begin{equation}
\begin{split}  
 0 \leq a_i \leq e,\ 
 &0 \leq b_i \leq e,\\ 
 v_u \bigl( 
 \phi &(v_i )u^{b_i} -v_{i+1} u^{a_i} 
 \bigr) \geq \max \{ 
 0,a_i +b_i -e \}
 \textrm{ for } 2\leq i \leq n. 
\end{split} \tag{$C_2$} \label{eq:C2}
\end{equation}

We show the following fact: 
\begin{equation}
\begin{split}
 &\textrm{If } 
 \mathscr{GR}_{V_{\F},0,\underline{a},\underline{b}}(\F') 
 \neq \emptyset , 
 \textrm{ there does not exist }
 (r_i ')_{1 \leq i \leq n} \in \mathbb{Z}^n \\ 
 &\textrm{ such that } 
 a_1 =b_1 -pr_1 ' -r_2 ' 
 \textrm{ and } 
 a_i -r_{i+1} ' =b_i -pr_i ' 
 \textrm{ for } 
 2\leq i \leq n. 
\end{split} \tag{$\diamondsuit$}
\end{equation}
We assume that 
there exists 
$(r_i ')_{1 \leq i \leq n} \in \mathbb{Z}^n$ 
satisfying this condition. 
Changing the basis of 
$\mathfrak{M} _{\underline{a},\underline{b},\F'}$ 
by 
$\Biggl(
 \begin{pmatrix}
 1 & u^{-r_i '} \\
 0 & 1 
 \end{pmatrix}
 \Biggr)_i $, 
we get 
\[
 M_{\F'} \sim 
 \Biggl( 
 \alpha_1
 \begin{pmatrix}
 u^{b_1 -pr_1 '} & 0 \\ 
 u^{b_1} & -u^{b_1 -r_2 '} 
 \end{pmatrix}, 
 \alpha_i
 \begin{pmatrix}
 u^{a_i} & 0 \\ 
 0 & u^{b_i}
 \end{pmatrix}_{2 \leq i \leq n} 
 \Biggr). 
\]
This contradicts that $V_{\F}$ 
is absolutely irreducible. 

\begin{lem}\label{range}
If 
$\mathscr{GR}_{V_{\F},0,\underline{a},\underline{b}}(\F')
 \neq \emptyset$, 
then
\[
 -\frac{e}{p-1} \leq a_1 \leq e,\ 
 0 \leq b_1 \leq \frac{pe}{p-1} 
 \textrm{ and }  
 0 \leq a_i ,b_i \leq e 
 \textrm{ for } 2 \leq i \leq n. 
\] 
\end{lem}
\begin{proof}
We take 
$\mathfrak{M} _{\F'} \in 
 \mathscr{GR}_{V_{\F},0,\underline{a},\underline{b}}(\F')$ and 
write it as 
$\Biggl(
 \begin{pmatrix}
 1 & v_i \\ 0 & 1
 \end{pmatrix} 
 \Biggr)_i \cdot
 \mathfrak{M} _{\underline{a},\underline{b},\F'}$ 
for some $(v_i )_{1 \leq i \leq n} \in \F'((u))^n$. 
We put $r_i =-v_u (v_i)$. 
We may assume $r_i \geq 0$ by Lemma \ref{equiv}. 

If $r_2 >e/(p-1)$, 
we have that 
$a_i -r_{i+1} =b_i -pr_i <0$ for $2 \leq i \leq n$ 
and $r_i >e/(p-1)$ for all $i$ 
by the condition (\ref{eq:C2}), 
and that 
$a_1 =b_1 -pr_1 -r_2 <0$ 
by the condition (\ref{eq:C1}). 
This contradicts $(\diamondsuit)$, 
and we have $r_2 \leq e/(p-1)$. 

Then $(C_1)$ and $(C_2)$ shows 
the claim. 
\end{proof}

To examine 
$|\mathscr{GR}_{V_{\F},0,\underline{a},\underline{b}}(\F')|$, 
we consider the case where 
$0 \leq a_1 \leq e$ and $0 \leq b_1 \leq e$, 
and the case where $\max \{-a_1 ,b_1 -e\}>0$. 

First, we treat the case 
where $0 \leq a_1 \leq e$ and $0 \leq b_1 \leq e$. 
In this case, the condition 
$u^e \mathfrak{M}_{\F'} \subset 
(1\otimes \phi ) \bigl( \phi ^* (\mathfrak{M}_{\F'}) \bigr)
\subset \mathfrak{M}_{\F'}$ 
is equivalent to the condition that 
$\max \{ pr_1 +r_2 ,pr_1 ,r_2 \} \leq \min \{e-a_1 ,b_1\}$ 
and (\ref{eq:C2}). 
We put 
\[
 I_{\underline{a},\underline{b}} = 
 \bigl\{ (R_1 ,R_2 ) \in \mathbb{Z} \times \mathbb{Z} 
 \bigm| pR_1 +R_2 \leq \min \{ e-a_1 ,b_1 \} ,\ 
 R_1 ,R_2 \geq 0 
 \bigr\} 
\] 
and 
\begin{align*}
 \mathscr{GR}_{V_{\F},0,\underline{a},\underline{b},R_1 ,R_2}(\F') &= 
 \Biggl\{ \Biggl(
 \begin{pmatrix}
 1 & v_i \\ 0 & 1
 \end{pmatrix} 
 \Biggr)_i \cdot
 \mathfrak{M} _{\underline{a},\underline{b},\F'} 
 \in 
 \mathscr{GR}_{V_{\F},0,\underline{a},\underline{b}}(\F')
 \Biggm| 
 v_i \in \F'((u)), \\ 
 & \hspace*{18em}
 r_1 = R_1 ,\ 
 r_2 = R_2 
 \Biggr\}
\end{align*}
for $(R_1 ,R_2) \in I_{\underline{a},\underline{b}}$. 
Then we have a disjoint union 
\[
 \mathscr{GR}_{V_{\F},0,\underline{a},\underline{b}}(\F') =
 \bigcup_{(R_1 ,R_2) \in I_{\underline{a},\underline{b}}} 
 \mathscr{GR}_{V_{\F},0,\underline{a},\underline{b},R_1 ,R_2}(\F') 
\]
by Lemma \ref{equiv}. 

We fix $(R_1 ,R_2) \in I_{\underline{a},\underline{b}}$. 
Then the condition that $r_1 = R_1$ and 
$r_2 = R_2$ implies 
$\max \{ pr_1 +r_2 ,pr_1 ,r_2 \} \leq \min \{e-a_1 ,b_1 \}$. 
So 
$\Biggl(
 \begin{pmatrix}
 1 & v_i \\ 0 & 1
 \end{pmatrix} 
 \Biggr)_i \cdot
 \mathfrak{M} _{\underline{a},\underline{b},\F'}$ 
gives a point of 
$\mathscr{GR}_{V_{\F},0,\underline{a},\underline{b},R_1 ,R_2}(\F')$ 
if and only if 
\[
 \max \{r_1 ,0\} = R_1 ,\ \max \{r_2 ,0\} = R_2 \textrm{ and (\ref{eq:C2})}. 
\] 
We assume 
$\mathscr{GR}_{V_{\F},0,\underline{a},\underline{b},R_1 ,R_2}(\F') 
 \neq \emptyset$. 
Considering 
$-v_u (v_i)$ for $(v_i )_{1 \leq i \leq n}$ 
that gives a point of 
$\mathscr{GR}_{V_{\F},0,\underline{a},\underline{b},R_1 ,R_2}(\F')$, 
we have the following two cases: 
\begin{enumerate}
 \item[(i)] 
 There are $2 \leq n_2 < n_1 \leq n+1$ and 
 $R_i \in \mathbb{Z}$ for 
 $3 \leq i \leq n_2$ and $n_1 \leq i \leq n$ 
 such that 
 \[
  a_i -R_{i+1} =b_i -pR_i 
  < \max \{0,a_i +b_i -e\}
 \]
 for $2 \leq i \leq n_2 -1$ and $n_1 \leq i \leq n$, 
 and 
 \[
  R_{n_1} \leq \min \{a_{n_1 -1} ,e-b_{n_1 -1} \},\ 
  R_{n_2} \leq \min \biggl\{ 
  \frac{e-a_{n_2}}{p} ,\frac{b_{n_2}}{p} 
  \biggr\}. 
 \]
 \item[(ii)]
 There are $R_i \in \mathbb{Z}$ for 
 $3 \leq i \leq n$ 
 such that 
 \[
  a_i -R_{i+1} =b_i -pR_i 
  < \max \{0,a_i +b_i -e\}
 \]
 for $2 \leq i \leq n$. 
\end{enumerate}
We note that (ii) includes the case $n=1$. 

We define an $\F'$-vector space 
$\widetilde{N}_{\underline{a},\underline{b},R_1 ,R_2,\F'}$ by 
\[
 \widetilde{N}_{\underline{a},\underline{b},R_1 ,R_2,\F'} = 
 \bigl\{
 (v_i)_{1 \leq i \leq n} \in \F'((u))^n \bigm| 
 r_1 \leq R_1 ,\ r_2 \leq R_2 \textrm{ and (\ref{eq:C2})}
 \bigr\}. 
\] 
We note that 
$\widetilde{N}_{\underline{a},\underline{b},R_1 ,R_2,\F'}
 \supset \F'[[u]]^n$. 
We put 
$N_{\underline{a},\underline{b},R_1 ,R_2,\F'} = 
 \widetilde{N}_{\underline{a},\underline{b},R_1 ,R_2,\F'}
 \big/ \F'[[u]]^n$
and 
$d_{\underline{a},\underline{b},R_1 ,R_2} =
 \dim_{\F'} N_{\underline{a},\underline{b},R_1 ,R_2,\F'}$. 
We note that 
$\dim_{\F'} N_{\underline{a},\underline{b},R_1 ,R_2,\F'}$ 
is independent of finite extensions $\F'$ of $\F$. 
We put 
\[
 \widetilde{N}_{\underline{a},\underline{b},R_1 ,R_2,\F'} ^{\circ} =
 \Bigl\{ (v_i )_{1 \leq i \leq n} \in 
 \widetilde{N}_{\underline{a},\underline{b},R_1 ,R_2,\F'} \Bigm| 
 r_1 = R_1 ,\ r_2 = R_2
 \Bigr\}. 
\] 
Let 
$N_{\underline{a},\underline{b},R_1 ,R_2,\F'} ^{\circ}$ 
be the image of 
$\widetilde{N}_{\underline{a},\underline{b},R_1 ,R_2,\F'} ^{\circ}$ 
in $N_{\underline{a},\underline{b},R_1 ,R_2,\F'}$. 
Then we have a bijection 
\[
 N_{\underline{a},\underline{b},R_1 ,R_2,\F'} ^{\circ} \to 
 \mathscr{GR}_{V_{\F},0,\underline{a},\underline{b},R_1 ,R_2}(\F')
\] 
by Lemma \ref{equiv}. 
By choosing a basis of 
$N_{\underline{a},\underline{b},R_1 ,R_2,\F}$ over $\F$, 
we have a morphism 
\[
 f_{\underline{a},\underline{b},R_1 ,R_2} : 
 \mathbb{A} ^{d_{\underline{a},\underline{b},R_1 ,R_2}} _{\F} 
 \to 
 \mathscr{GR}_{V_{\F},0}
\] 
in the case $R_1 =R_2 =0$, 
\[
 f_{\underline{a},\underline{b},R_1 ,R_2} : 
 \mathbb{A} 
 ^{(d_{\underline{a},\underline{b},R_1 ,R_2} -2)} _{\F} 
 \times \mathbb{G} _{m,\F} ^2 \to 
 \mathscr{GR}_{V_{\F},0}
\] 
in the case where 
$R_1 >0$, $R_2 >0$ and (i) holds true, and 
\[
 f_{\underline{a},\underline{b},R_1 ,R_2} : 
 \mathbb{A} ^{(d_{\underline{a},\underline{b},R_1 ,R_2} -1)} _{\F} 
 \times \mathbb{G} _{m, \F} \to 
 \mathscr{GR}_{V_{\F},0}
\] 
in the other case, 
such that 
$f_{\underline{a},\underline{b},R_1 ,R_2}(\F')$ 
is injective and 
the image of 
$f_{\underline{a},\underline{b},R_1 ,R_2}(\F')$ 
is $\mathscr{GR}_{V_{\F},0,\underline{a},\underline{b},R_1 ,R_2}(\F')$. 

\begin{lem}\label{nonnegative}
If $0 \leq a_1 \leq e$ and $0 \leq b_1 \leq e$, 
the followings hold: 
\begin{enumerate}
\item[(a)] 
In the case $e_1 =0$, 
we have $d_{\underline{a},\underline{b},R_1 ,R_2} \leq ne_0 -1$. 
In this case, if 
$a_1 =0$, $b_1 =(p+1)e_0 -1$, $a_i =e_0$ and 
$b_i =pe_0$ for $2 \leq i \leq n$, 
then 
there exists 
$(R_1 ,R_2 ) \in I_{\underline{a},\underline{b}}$ such that 
$d_{\underline{a},\underline{b},R_1 ,R_2} = ne_0 -1$. 
\item[(b)] 
In the case $1 \leq e_1 \leq p-1$, 
we have $d_{\underline{a},\underline{b},R_1 ,R_2} \leq ne_0$. 
In this case, if 
$a_1 =0$, $b_1 =(p+1)e_0 +1$, 
$a_i =e_0$ and $b_i =pe_0$ for 
$2 \leq i \leq n$,
then 
there exists 
$(R_1 ,R_2 ) \in I_{\underline{a},\underline{b}}$ such that 
$d_{\underline{a},\underline{b},R_1 ,R_2} = ne_0$. 
\item[(c)] 
In the case $e_1 = p$, 
we have 
$d_{\underline{a},\underline{b},R_1 ,R_2} \leq ne_0 +[n/2]$. 
In this case, if 
$a_1 =0$, $b_1 =(p+1)e_0 +1$, 
$a_i =2e_0 +1 -e_{0,i}$ and 
$b_i =pe_{0,i}$ for 
$2 \leq i \leq n$, 
then 
there exists 
$(R_1 ,R_2 ) \in I_{\underline{a},\underline{b}}$ such that 
$d_{\underline{a},\underline{b},R_1 ,R_2} = ne_0 +[n/2]$. 
Here, $e_{0,i} =e_0$ 
if $i$ is odd, 
and $e_{0,i} =e_0 +1$ 
if $i$ is even. 
\end{enumerate}
\end{lem}
\begin{proof}
First, we treat the case $n=1$. 
In this case, we have 
\[
 R_1 =R_2 \leq \biggl[
 \frac{\min \{e-a_1 ,b_1\} }{p+1}\biggr]
 \leq e_0. 
\]
So we get 
$d_{\underline{a},\underline{b},R_1 ,R_2} \leq e_0$ 
for 
$(\underline{a},\underline{b}) \in 
 \mathbb{Z}^n \times \mathbb{Z}^n$ and 
$(R_1 ,R_2) \in I_{\underline{a},\underline{b}}$ 
such that 
$\mathscr{GR}_{V_{\F},0,\underline{a},\underline{b},R_1 ,R_2}(\F')
 \neq \emptyset$ 
and $0 \leq a_1 ,b_1 \leq e$. 
We have to eliminate the possibility of equality 
in the case $e_1 =0$. 
In this case, if we have 
$d_{\underline{a},\underline{b},R_1 ,R_2} = e_0$, 
then $a_1 =0$ and $b_1 =(p+1)e_0$. 
This contradicts $(\diamondsuit )$. 

We can check that if 
$e_1 =0$, $a_1 =0$, $b_1 =e-1$ and 
$R_1 =R_2 =e_0 -1$, 
then 
$d_{\underline{a},\underline{b},R_1 ,R_2} = e_0 -1$, 
and that if 
$e_1 \neq 0$, $a_1 =0$, $b_1 =(p+1)e_0 +1$ and 
$R_1 =R_2 =e_0$, 
then 
$d_{\underline{a},\underline{b},R_1 ,R_2} = e_0$. 

So we may assume $n \geq 2$. 
We put
\begin{align*}
 S_{\underline{a},\underline{b},R_1 ,R_2 ,1} &=\bigl\{
 (u^{-r_1} ,0,\ldots ,0) \in \F((u))^n 
 \bigm| 1 \leq r_1 \leq 
 \min \{
 R_1 ,a_{n} ,e-b_{n} 
 \} \bigr\},\\
 S_{\underline{a},\underline{b},R_1 ,R_2 ,2} &=\Biggl\{
 (0,u^{-r_2} ,0, \ldots ,0) \in \F((u))^n 
 \Biggm| 1 \leq r_2 \leq 
 \min \biggl\{
 R_2 ,\frac{e-a_2}{p},\frac{b_2}{p} 
 \biggr\} \Biggr\},\\ 
 S_{\underline{a},\underline{b},R_1 ,R_2 ,i} &=\Biggl\{
 (0,\ldots ,0,v_i ,0, \ldots ,0) \in \F((u))^n 
 \Biggm|  v_i =u^{-r_i}, \\
 &\hspace{12em} 1 \leq r_i \leq 
 \min \biggl\{
 a_{i-1} ,e-b_{i-1} ,\frac{e-a_i}{p},\frac{b_i}{p} 
 \biggr\} \Biggr\} \\
\intertext{for $3 \leq i \leq n$, and}
 S_{\underline{a},\underline{b},R_1 ,R_2 ,i,j} 
 &=\Biggl\{
 (0,\ldots ,0,v_i ,v_{i+1}, \ldots ,v_{j+1} ,0,\ldots ,0) \in \F((u))^n 
 \Biggm| v_i =u^{-r_i},\ \\ 
 & 
 r_i \leq \min \{a_{i-1} ,e-b_{i-1} \}
 \textrm{ if } i \neq 2,\ 
 r_2 \leq R_2 \textrm{ if } i=2, \\
 & 
 u^{a_{l}} v_{l+1} = 
 u^{b_{l}} \phi(v_{l}) 
 \textrm{ and }
 -\!v_u (v_{l+1} ) > 
 \min \{a_{l} ,e-b_{l} \} 
 \textrm{ for } 
 i \leq l \leq j, \\
 & 
 -\!v_u (v_{j+1}) \leq 
 \min \biggl\{
 \frac{e-a_{j+1}}{p}, \frac{b_{j+1}}{p} 
 \biggr\} \textrm{ if } j \neq n,\ 
 -v_u (v_1) \leq R_1 \textrm{ if } j=n 
 \Biggr\} 
\end{align*}
for $2 \leq i \leq j \leq n$. 
In the above definitions, $v_i$ is on the $i$-th component. 
Then, as in the proof of Lemma \ref{basis}, 
we can check that 
$\bigcup_{i} S_{\underline{a},\underline{b},R_1 ,R_2 ,i} 
 \cup \bigcup_{i,j} S_{\underline{a},\underline{b},R_1 ,R_2 ,i,j}$ 
is an $\F$-basis of $N_{\underline{a},\underline{b},R_1 ,R_2,\F}$. 
So we have 
$d_{\underline{a},\underline{b},R_1 ,R_2} =
 \sum_i |S_{\underline{a},\underline{b},R_1 ,R_2 ,i}| + 
 \sum_{i,j} |S_{\underline{a},\underline{b},R_1 ,R_2 ,i,j}|$. 

We put 
\begin{align*}
 T_{\underline{a},\underline{b},R_1 ,R_2 ,1} &=\bigl\{
 m \in \mathbb{Z} \bigm| 
 \min \{a_n ,e-b_n \} < pm+a_n -b_n \leq R_1 
 \bigr\}, \\
\intertext{$T_{\underline{a},\underline{b},R_1 ,R_2 ,2} =\emptyset$ and}
 T_{\underline{a},\underline{b},R_1 ,R_2 ,i} &=\Biggl\{
 m \in \mathbb{Z} \Biggm| 
 \min \{a_{i-1} ,e-b_{i-1} \} < pm+a_{i-1} -b_{i-1}\\ 
 & \hspace*{14.25em}
 \leq 
 \min \biggl\{
 \frac{e-a_i }{p}, \frac{b_i }{p} 
 \biggr\} \Biggr\} 
\end{align*}
for $3 \leq i \leq n$. 
We consider the map
\[
 \bigcup_{2 \leq i \leq h-1} 
 S_{\underline{a},\underline{b},R_1 ,R_2 ,i,h-1} 
 \to T_{\underline{a},\underline{b},R_1 ,R_2 ,h} ;\ 
 (v_{i'})_{1 \leq i' \leq n} \mapsto -v_u (v_{h-1}) 
\]
for $3 \leq h \leq n+1$. 
We can easily check that this map is injective. 
So we have 
$\sum_{2 \leq i \leq h-1} |S_{\underline{a},\underline{b},R_1 ,R_2 ,i,h-1}| 
 \leq |T_{\underline{a},\underline{b},R_1 ,R_2 ,h}|$ 
and 
$d_{\underline{a},\underline{b},R_1 ,R_2} \leq 
 \sum_{1 \leq i \leq n} \bigl( 
 |S_{\underline{a},\underline{b},R_1 ,R_2 ,i}| + 
 |T_{\underline{a},\underline{b},R_1 ,R_2 ,i}| \bigr)$. 

We take 
$(\underline{a}',\underline{b}') 
 \in \mathbb{Z}^n \times \mathbb{Z}^n$ 
and $(R_1 ' ,R_2 ') \in I_{\underline{a}',\underline{b}'}$  
such that $0 \leq a_1 ',b_1 ' \leq e$ and 
$\sum_{1 \leq i \leq n} \bigl( 
 |S_{\underline{a}',\underline{b}',R_1 ' ,R_2 ' ,i}| +
 |T_{\underline{a}',\underline{b}',R_1 ' ,R_2 ' ,i}| \bigr)$
is the maximum. 
We can prove that 
$|T_{\underline{a}',\underline{b}',R_1 ' ,R_2 ' ,i}| \leq 1$ 
for all $i$ as in the proof 
of Lemma \ref{boundT}. 

We can also show that 
\begin{enumerate}
 \item[$(A_i)$] 
if 
$|S_{\underline{a}',\underline{b}',R_1 ' ,R_2 ' ,i}| +
 |T_{\underline{a}',\underline{b}',R_1 ' ,R_2 ' ,i}| 
 =e_0 +l$ 
for $l \geq 1$,\\ 
then 
$|S_{\underline{a}',\underline{b}',R_1 ' ,R_2 ' ,i+1}| +
 |T_{\underline{a}',\underline{b}',R_1 ' ,R_2 ' ,i+1}| 
 \leq e_0 +e_1 -pl+1$ 
\end{enumerate}
for $i \neq 1$, and that 
\begin{enumerate}
 \item[$(B_i)$] 
if 
$|S_{\underline{a}',\underline{b}',R_1 ' ,R_2 ' ,i}| +
 |T_{\underline{a}',\underline{b}',R_1 ' ,R_2 ' ,i}| 
 =e_0 +1$\\ 
and 
$|S_{\underline{a}',\underline{b}',R_1 ' ,R_2 ' ,i+1}| +
 |T_{\underline{a}',\underline{b}',R_1 ' ,R_2 ' ,i+1}| 
 = e_0 +e_1 -p+1$,\\ 
then 
$|S_{\underline{a}',\underline{b}',R_1 ' ,R_2 ' ,i+2}| +
 |T_{\underline{a}',\underline{b}',R_1 ' ,R_2 ' ,i+2}| 
 \leq e_0 -(p-1)e_1 +1$ 
\end{enumerate}
for $2 \leq i \leq n-1$ 
as in the proof of 
Lemma \ref{AB}. 
By the same argument, 
we can show that 
\begin{enumerate}
 \item[$(A_1)$] 
if 
$|S_{\underline{a}',\underline{b}',R_1 ' ,R_2 ' ,1}| +
 |T_{\underline{a}',\underline{b}',R_1 ' ,R_2 ' ,1}| 
 =e_0 +l$ 
for $l \geq 1$,\\ 
then 
$|S_{\underline{a}',\underline{b}',R_1 ' ,R_2 ' ,2}| +
 |T_{\underline{a}',\underline{b}',R_1 ' ,R_2 ' ,2}| 
 \leq e_0 +e_1 -pl$, 
\end{enumerate}
and that 
\begin{enumerate}
 \item[$(B_n)$] 
if 
$|S_{\underline{a}',\underline{b}',R_1 ' ,R_2 ' ,n}| +
 |T_{\underline{a}',\underline{b}',R_1 ' ,R_2 ' ,n}| 
 =e_0 +1$\\ 
and 
$|S_{\underline{a}',\underline{b}',R_1 ' ,R_2 ' ,1}| +
 |T_{\underline{a}',\underline{b}',R_1 ' ,R_2 ' ,1}| 
 = e_0 +e_1 -p+1$,\\ 
then 
$|S_{\underline{a}',\underline{b}',R_1 ' ,R_2 ' ,2}| +
 |T_{\underline{a}',\underline{b}',R_1 ' ,R_2 ' ,2}| 
 \leq e_0 -(p-1)e_1 $, 
\end{enumerate}
using the followings: 
\begin{align*}
 &|S_{\underline{a}',\underline{b}',R_1 ' ,R_2 ' ,1}| +
 |T_{\underline{a}',\underline{b}',R_1 ' ,R_2 ' ,1}| 
 \leq R_1,\ 
 pR_1 +R_2 \leq e,\\ 
 &|S_{\underline{a}',\underline{b}',R_1 ' ,R_2 ' ,2}| 
 \leq R_2 
 \textrm{ and } 
 T_{\underline{a}',\underline{b}',R_1 ' ,R_2 ' ,2} =\emptyset .
\end{align*}

Firstly, we treat the case where 
$0 \leq e_1 \leq p-1$, 
that is, (a) or (b). 
We note that 
$e_0 +e_1 -pl+1 \leq e_0 -p(l-1)-1$ 
in the case $0 \leq e_1 \leq p-2$, 
and that 
$e_0 +e_1 -pl+1 =e_0 -p(l-1)$ and 
$e_0 -(p-1)e_1 +1 \leq e_0 -3$ 
in the case $e_1 =p-1$. 
Then $(A_i)$ for all $i$ 
and $(B_i)$ for $i \neq 1$ 
implies 
\begin{align*}
 d _{\underline{a},\underline{b},R_1 ,R_2 } &\leq 
 \sum_{1\leq i \leq n} 
 \bigl( |S_{\underline{a},\underline{b},R_1 ,R_2  ,i}|+
 |T_{\underline{a},\underline{b},R_1 ,R_2  ,i}| \bigr) \\ 
 &\leq
 \sum_{1\leq i \leq n} 
 \bigl( |S_{\underline{a}',\underline{b}',R_1 ' ,R_2 ' ,i}|+
 |T_{\underline{a}',\underline{b}',R_1 ' ,R_2 ' ,i}| \bigr) 
 \leq ne_0 
\end{align*} 
for 
$(\underline{a},\underline{b}) \in 
 \mathbb{Z}^n \times \mathbb{Z}^n$ and 
$(R_1 ,R_2) \in I_{\underline{a},\underline{b}}$ 
such that 
$\mathscr{GR}_{V_{\F},0,\underline{a},\underline{b},R_1 ,R_2}(\F')
 \neq \emptyset$ 
and $0 \leq a_1 ,b_1 \leq e$. 
So we get the desired bound, 
if $1\leq e_1 \leq p-1$. 
In the case $e_1 =0$, 
we have to eliminate the possibility 
of equality. 
In this case, if we have equality, 
we get that 
$\sum_{1\leq i \leq n} 
 \bigl( |S_{\underline{a},\underline{b},R_1 ,R_2  ,i}|+
 |T_{\underline{a},\underline{b},R_1 ,R_2  ,i}| \bigr)$ 
is the maximum and 
$\bigl( |S_{\underline{a},\underline{b},R_1 ,R_2 ,i}|+
 |T_{\underline{a},\underline{b},R_1 ,R_2 ,i}| \bigr) =e_0$ 
for all $i$ by $(A_i)$ for all $i$. 
Then we have 
\[
 R_1 =R_2 =e_0 ,\ 
 e_0 -1 \leq a_i \leq e_0 ,\ 
 pe_0 \leq b_i \leq pe_0 +1 
\textrm{ for } 2 \leq i \leq n 
\]
by the followings: 
\begin{align*}
 &pR_1 +R_2 =e, 
 |S_{\underline{a},\underline{b},R_1 ,R_2 ,1}|+
 |T_{\underline{a},\underline{b},R_1 ,R_2 ,1}| \leq R_1 , 
 |S_{\underline{a},\underline{b},R_1 ,R_2 ,2}| \leq R_2 , \\
 &|S_{\underline{a},\underline{b},R_1 ,R_2 ,i}|+
 |T_{\underline{a},\underline{b},R_1 ,R_2 ,i}| 
 \leq \min \{(e-a_i )/p,b_i /p \} 
 \textrm{ for } 2 \leq i \leq n \\
 &\textrm{and } 
 |S_{\underline{a},\underline{b},R_1 ,R_2 ,i}| \geq e_0 -1 
 \textrm{ for } i\neq 2. 
\end{align*}
Now we have $a_1 =0$ and $b_1 =(p+1)e_0$ 
by $R_1 =R_2 =e_0$. 
We show that 
$|T_{\underline{a},\underline{b},R_1 ,R_2 ,i}| =0$ 
for $3 \leq i \leq n$. 
We assume that 
$|T_{\underline{a},\underline{b},R_1 ,R_2 ,i_0}| =1$ 
for some $i_0 \neq 1,2$, 
and let 
$m$ be the unique element of 
$T_{\underline{a},\underline{b},R_1 ,R_2 ,i_0}$. 
Then, by the definition of 
$T_{\underline{a},\underline{b},R_1 ,R_2 ,i_0}$, 
we have
\begin{align*}
 \min \biggl\{ \frac{e-a_{i_0}}{p}, 
 \frac{b_{i_0}}{p} \biggr\} - 
 \min \{a_{i_0 -1} ,e-b_{i_0 -1} \} 
 &\geq 
 pm -\min \{e-a_{i_0 -1} ,b_{i_0 -1} \}\\ 
 &\geq p-1 \geq 2, 
\end{align*} 
because 
$pe_0 \leq \min \{e-a_{i_0 -1} ,b_{i_0 -1} \} \leq pe_0 +1$ 
and $pm -\min \{e-a_{i_0 -1} ,b_{i_0 -1} \}>0$. 
This contradicts the possibilities of 
$a_{i_0 -1}$, $a_{i_0}$, 
$b_{i_0 -1}$ and $b_{i_0}$. 
The same argument shows that 
$|T_{\underline{a},\underline{b},R_1 ,R_2 ,1}| =0$. 
Now we have 
$|S_{\underline{a},\underline{b},R_1 ,R_2 ,i}| =e_0$ 
for all $i$, and that 
\[
 a_1 =0,\ b_1 =(p+1)e_0 ,\ 
 a_i =e_0 ,\ b_i =pe_0 
 \textrm{ for } 
 2 \leq i \leq n. 
\]
Then we have 
\[
 a_1 =b_1 -pr_1 ' -r_2 ' 
 \textrm{ and } 
 a_i -r_{i+1} ' =b_i -pr_i ' 
 \textrm{ for } 
 2\leq i \leq n 
\]
for 
$(r_i ')_{1 \leq i \leq n}=
 (e_0)_{1 \leq i \leq n}$. 
This contradicts $(\diamondsuit)$.  
So we have 
$d _{\underline{a},\underline{b},R_1 ,R_2 } \leq ne_0 -1$, 
if $e_1 =0$. 

We can check that 
if $e_1 =0$, 
$a_1 =0$, $b_1 =(p+1)e_0 -1$, 
$R_1 =e_0$, $R_2 =e_0 -1$, 
$a_i =e_0$ and $b_i =pe_0$ for 
$2 \leq i \leq n$, 
then 
$d _{\underline{a},\underline{b},R_1 ,R_2 } 
 \geq 
 \sum_{1\leq i \leq n} 
 |S_{\underline{a},\underline{b},R_1 ,R_2  ,i}| =
 ne_0 -1$. 
We can check also that 
if $1 \leq e_1 \leq p-1$, 
$a_1 =0$, $b_1 =(p+1)e_0 +1$, 
$R_1 =e_0$, $R_2 =e_0 +1$, 
$a_i =e_0$ and $b_i =pe_0$ for 
$2 \leq i \leq n$, then 
$d _{\underline{a},\underline{b},R_1 ,R_2 } 
 \geq 
 \sum_{1\leq i \leq n} 
 |S_{\underline{a},\underline{b},R_1 ,R_2  ,i}| =
 ne_0$. 

Secondly, we treat (c). 
In this case, 
we note that 
$e_0 +e_1 -pl+1 = e_0 -p(l-1)+1$ and 
$e_0 -(p-1)e_1 +1 \leq e_0 -5$. 
Then $(A_i)$ for all $i$ 
and $(B_i)$ for $i \neq 1$ 
implies 
\begin{align*} 
 d _{\underline{a},\underline{b},R_1 ,R_2 } &\leq 
 \sum_{1\leq i \leq n} 
 \bigl( |S_{\underline{a},\underline{b},R_1 ,R_2  ,i}|+
 |T_{\underline{a},\underline{b},R_1 ,R_2  ,i}| \bigr) \\
 &\leq 
 \sum_{1\leq i \leq n} 
 \bigl( |S_{\underline{a}',\underline{b}',R_1 ' ,R_2 ' ,i}|+
 |T_{\underline{a}',\underline{b}',R_1 ' ,R_2 ' ,i}| \bigr) 
 \leq ne_0 +
 \biggl[
 \frac{n}{2}
 \biggr]
\end{align*} 
for 
$(\underline{a},\underline{b}) \in 
 \mathbb{Z}^n \times \mathbb{Z}^n$ and 
$(R_1 ,R_2) \in I_{\underline{a},\underline{b}}$ 
such that $0 \leq a_1 ,b_1 \leq e$. 
So we get the desired bound. 

We can check that 
if $e_1 =p$, 
$a_1 =0$, $b_1 =(p+1)e_0 +1$, 
$R_1 =e_0$, $R_2 =e_0 +1$, 
$a_i =2e_0 +1 -e_{0,i}$ and 
$b_i =pe_{0,i}$ for 
$2 \leq i \leq n$, 
then 
$d _{\underline{a},\underline{b},R_1 ,R_2 } 
 \geq 
 \sum_{1\leq i \leq n} 
 |S_{\underline{a},\underline{b},R_1 ,R_2  ,i}| =
 ne_0 +[n/2]$.
\end{proof}

Next, we consider the remaining case, 
that is, the case where 
$\max \{-a_1 ,b_1 -e\} > 0$. 
In this case, 
$v_u \bigl( 
 u^{a_1} -\phi(v_1 )v_2 u^{b_1} \bigr)
 \geq \max \{ 0,a_1 +b_1 -e \}$ 
implies $pr_1 +r_2 =b_1 -a_1$, 
because $a_1 <\max \{ 0,a_1 +b_1 -e \}$. 
So the condition 
$u^e \mathfrak{M}_{\F'} \subset 
(1\otimes \phi ) \bigl( \phi ^* (\mathfrak{M}_{\F'}) \bigr)
\subset \mathfrak{M}_{\F'}$ 
implies 
\[
 pr_1 +r_2 =b_1 -a_1 ,\ 
 \max\{-a_1 ,b_1 -e\} \leq r_2 
 \leq \min \{ e-a_1 ,b_1 \}. 
\]
We note that if $n=1$, then 
$pr_1 +r_2 =b_1 -a_1$ contradicts 
$(\diamondsuit)$ because 
$r_1 =r_2$. 
So we may assume $n \geq 2$. 
We put 
\begin{align*}
 I_{\underline{a},\underline{b}} &= 
 \Bigl\{ (R_1 ,R_2 ) \in \mathbb{Z} \times \mathbb{Z} 
 \Bigm| pR_1 +R_2 =b_1 -a_1 ,\\ 
 & \hspace*{10.1em}
 \max \{ -a_1 ,b_1 -e\} 
 \leq R_2 \leq \min \{ e-a_1 ,b_1 \} 
 \Bigr\} 
\end{align*} 
and 
$m_{\underline{a},\underline{b}} =
 \bigl[\bigr(\max \{-a_1 ,b_1 -e \} -1\bigr) \big/p \bigr]$. 
We note that 
$R_1 \geq m_{\underline{a},\underline{b}} +1>0$ 
and 
$R_2 \geq \max \{ -a_1 ,b_1 -e\} >0$. 
We put 
\begin{align*}
 \mathscr{GR}_{V_{\F},0,\underline{a},\underline{b},R_1 ,R_2}(\F') &= 
 \Biggl\{ \Biggl(
 \begin{pmatrix}
 1 & v_i \\ 0 & 1
 \end{pmatrix} 
 \Biggr)_i \cdot
 \mathfrak{M} _{\underline{a},\underline{b},\F'} 
 \in 
 \mathscr{GR}_{V_{\F},0,\underline{a},\underline{b}}(\F')
 \Biggm| 
 v_i \in \F'((u)), \\ 
 & \hspace*{13em}
 v_u (v_1) = -R_1 ,\ 
 v_u (v_2) = -R_2 
 \Biggr\}
\end{align*}
for $(R_1 ,R_2) \in I_{\underline{a},\underline{b}}$. 
Then we have a disjoint union 
\[
 \mathscr{GR}_{V_{\F},0,\underline{a},\underline{b}}(\F') =
 \bigcup_{(R_1 ,R_2) \in I_{\underline{a},\underline{b}}} 
 \mathscr{GR}_{V_{\F},0,\underline{a},\underline{b},R_1 ,R_2}(\F') 
\] 
by Lemma \ref{equiv}. 
Extending the field $\F$, 
we may assume that 
$\mathscr{GR}_{V_{\F},0,\underline{a},\underline{b},R_1 ,R_2}(\F')
 \neq \emptyset$ 
if and only if 
$\mathscr{GR}_{V_{\F},0,\underline{a},\underline{b},R_1 ,R_2}(\F)
 \neq \emptyset$ 
for each 
$(R_1 ,R_2) \in I_{\underline{a},\underline{b}}$, 
$(\underline{a},\underline{b}) \in 
 \mathbb{Z}^n \times \mathbb{Z}^n$ 
and any finite extension $\F'$ of $\F$. 

We fix $(R_1 ,R_2) \in I_{\underline{a},\underline{b}}$, 
and assume 
$\mathscr{GR}_{V_{\F},0,\underline{a},\underline{b},R_1 ,R_2}(\F)
 \neq \emptyset$. 
If $v_u (v_1 )=-R_1$ and 
$v_u (v_2 )=-R_2$, 
the condition
$v_u \bigl( 
 u^{a_1} -\phi(v_1 )v_2 u^{b_1} \bigr)
 \geq \max \{ 0,a_1 +b_1 -e \}$ 
is equivalent to the following:
\begin{center}
There uniquely exist 
$\gamma_{1,0} ,\gamma_{2,0} \in (\F')^{\times}$ and 
$\gamma_{1,i} ,\gamma_{2,i} \in \F'$ 
for $1 \leq i \leq 
 m_{\underline{a},\underline{b}}$ 
such that  
\begin{align*}
 & -\!v_u \Biggl( v_1 -\sum_{0 \leq i \leq 
 m_{\underline{a},\underline{b}} }\gamma_{1,i} u^{-R_1 +i}
 \Biggr) \leq R_1 -m_{\underline{a},\underline{b}} -1,\\ 
 & -\!v_u \Biggl( v_2 -\sum_{0 \leq i \leq 
 m_{\underline{a},\underline{b}} }\gamma_{2,i} u^{-R_2 +pi}
 \Biggr) \leq R_2 -\max \{-a_1 ,b_1 -e\}, \\ 
 &\ \gamma_{1,0} \gamma_{2,0} =1,\ 
 \sum_{0\leq i \leq l} 
 \gamma_{1,i} \gamma_{2,l-i} =0 
 \textrm{ for } 1 \leq l \leq 
 m_{\underline{a},\underline{b}}. 
\end{align*}
\end{center}
We note that 
$(\gamma_{1,i})_{0 \leq i \leq
 m_{\underline{a},\underline{b}}}$ 
determines 
$(\gamma_{1,i},\gamma_{2,i})_{0 \leq i \leq
 m_{\underline{a},\underline{b}}}$. 

We prove that 
for $0 \leq i \leq m_{\underline{a},\underline{b}}$ 
there uniquely exist 
$2 \leq n_{2,i} < n_{1,i} \leq n+1$, 
$r_{1,i,j} \in \mathbb{Q}$ for 
$n_{1,i} \leq j \leq n+1$ and 
$r_{2,i,j} \in \mathbb{Z}$ for 
$2 \leq j \leq n_{2,i}$ 
such that 
$r_{1,0,j} \in \mathbb{Z}$ for 
$n_{1,0} \leq j \leq n+1$ and 
\begin{align*}
 a_j -r_{1,i,j+1} &=b_j -pr_{1,i,j} <
 \max \{0,a_j +b_j -e\} 
 \textrm{ for } 
 n_{1,i} \leq j \leq n,\\
 r_{1,i,n+1} &=R_1 -i,\ 
 r_{1,i,n_{1,i}} \leq 
 \min \{ a_{n_{1,i} -1},e-b_{n_{1,i} -1} \},\\ 
 a_j -r_{2,i,j+1} &=b_j -pr_{2,i,j} <
 \max \{0,a_j +b_j -e\} 
 \textrm{ for } 
 2 \leq j \leq n_{2,i} -1,\\ 
 r_{2,i,2} &=R_2 -pi,\ 
 r_{2,i,n_{2,i}} \leq 
 \min \biggl\{ \frac{e-a_{n_{2,i}}}{p},\frac{b_{n_{2,i}}}{p} \biggr\}. 
\end{align*} 
Define 
$r_{1,i,j} \in \mathbb{Q}$ for 
$2 \leq j \leq n+1$ and 
$r_{2,i,j} \in \mathbb{Z}$ for 
$2 \leq j \leq n+1$ 
such that 
\begin{align*}
r_{1,i,n+1} &=R_1 -i,\ 
a_j -r_{1,i,j+1} =b_j -pr_{1,i,j} 
\textrm{ for } 
2 \leq j \leq n,\\ 
r_{2,i,2} &=R_2 -pi,\ 
a_j -r_{2,i,j+1} =b_j -pr_{2,i,j} 
\textrm{ for } 
2 \leq j \leq n. 
\end{align*}
We put
\begin{align*}
 n_{1,i} &=\max \Bigl\{
 \bigl\{ 3 \leq j \leq n+1 \bigm| 
 r_{1,i,j} \leq \min \{
 a_{j-1} ,e-b_{j-1} \}
 \bigr\} \cup \{2\}
 \Bigr\}, \\ 
 n_{2,i} &=\min \Biggl\{
 \biggl\{ 2 \leq j \leq n \biggm| 
 r_{2,i,j} \leq \min \Bigl\{
 \frac{e-a_j }{p} ,\frac{b_j }{p} 
 \Bigr\}
 \biggr\} \cup \{n+1\}
 \Biggr\}. 
\end{align*}
We consider $(v_i)_{1 \leq i \leq n}$ that gives 
a point of 
$\mathscr{GR}_{V_{\F},0,\underline{a},\underline{b},R_1 ,R_2}(\F)$. 
Then we have 
$r_{1,0,j} =-v_u (v_j) \in \mathbb{Z}$ for 
$n_{1,0} \leq j \leq n+1$ and 
$r_{2,0,j} =-v_u (v_j) \in \mathbb{Z}$ for 
$2 \leq j \leq n_{2,0}$. 
It remains to show that 
$n_{2,i} <n_{1,i}$. 
We have 
$n_{2,i} \leq n_{2,0}$ and 
$n_{1,0} \leq n_{1,i}$, 
because 
$r_{1,i,j} \leq r_{1,0,j}$ and 
$r_{2,i,j} \leq r_{2,0,j}$ 
for $2 \leq j \leq n+1$. 
So it suffices to show 
$n_{2,0,j} < n_{1,0,j}$. 
If $n_{2,0,j} \geq n_{1,0,j}$, we have 
\[
 a_1 =b_1 -pv_u (v_1) -v_u (v_2) 
 \textrm{ and } 
 a_j -v_u (v_{j+1} )=b_j -v_u (v_j) 
 \textrm{ for } 
 2 \leq j \leq n, 
\]
and this contradicts $(\diamondsuit)$. 

We put 
\[
 M_{\underline{a},\underline{b},R_1 ,R_2} = 
 \bigl\{ 0 \leq i \leq m_{\underline{a},\underline{b}} 
 \bigm| r_{1,i,j} \in \mathbb{Z} \textrm{ for } 
 n_{1,i} \leq j \leq n+1 
 \bigr\}.
\]
For $(v_i )_{1 \leq i \leq n}$ that gives 
a point of 
$\mathscr{GR}_{V_{\F},0,\underline{a},\underline{b},R_1 ,R_2}(\F')$, 
we take $\gamma_{1,i}$, $\gamma_{2,i}$ and 
$n_{1,i}$, $n_{2,i}$, $r_{1,i,j}$, $r_{2,i,j}$ as above. 
We note that 
$\gamma_{1,i} =0$ if 
$i \notin M_{\underline{a},\underline{b},R_1 ,R_2}$. 
We put 
\begin{align*}
 M_{1,\underline{a},\underline{b},R_1 ,R_2 ,j} &= 
 \bigl\{ 0 \leq i \leq m_{\underline{a},\underline{b}} 
 \bigm| n_{1,i} \leq j \leq n+1 \bigr\},\\
 M_{2,\underline{a},\underline{b},R_1 ,R_2 ,j} &= 
 \bigl\{ 0 \leq i \leq m_{\underline{a},\underline{b}} 
 \bigm| 2 \leq j \leq n_{2,i} \bigr\} 
\end{align*}
for $2 \leq j \leq n+1$, 
and define 
$(v_i ^{\ast} )_{1 \leq i \leq n} \in \F' ((u))^n$ 
by 
\[
 v_j ^{\ast} =v_j -
 \sum_{i \in M_{1,\underline{a},\underline{b},R_1 ,R_2 ,j}} 
 \gamma_{1,i} u^{-r_{1,i,j}} -
 \sum_{i \in M_{2,\underline{a},\underline{b},R_1 ,R_2 ,j}}
 \gamma_{2,i} u^{-r_{2,i,j}} 
\] 
for $2 \leq j \leq n+1$. 
This is well-defined by the above remark. 
We put 
\begin{align*}
 \widetilde{N}_{\underline{a},\underline{b},R_1 ,R_2,\F'} ^{\ast} &= 
 \bigl\{
 (v_i ^{\ast})_{1 \leq i \leq n} \in \F'((u))^n \bigm| 
 (v_i )_{1 \leq i \leq n} \in \F'((u))^n  
 \textrm{ gives} \\
 & \hspace{12em}
 \textrm{a point of } 
 \mathscr{GR}_{V_{\F},0,\underline{a},\underline{b},R_1 ,R_2}(\F') 
 \bigr\}. 
\end{align*}
Then we have 
\begin{align*}
 \widetilde{N}_{\underline{a},\underline{b},R_1 ,R_2,\F'} ^{\ast} = 
 \bigl\{
 (v_i )_{1 \leq i \leq n} \in \F'((u))^n \bigm| 
 & -\!v_u (v_1 ) \leq R_1 -m_{\underline{a},\underline{b}} -1, \\ 
 & -\!v_u (v_2 ) \leq R_2 - \max \{-a_1 ,b_1 -e\}, 
 (C_2)
 \bigr\}
\end{align*}
by the construction of $(v_i ^{\ast})_{1 \leq i \leq n}$ and 
the conditions $(C_1 )$ and $(C_2 )$. 
This implies that 
$\widetilde{N}_{\underline{a},\underline{b},R_1 ,R_2,\F'} ^{\ast}
 \subset \F'((u))^n$ 
is an $\F'$-vector subspace, 
and 
$\widetilde{N}_{\underline{a},\underline{b},R_1 ,R_2,\F'} ^{\ast}
 \supset \F'[[u]]^n$. 

We put 
\[
 N_{\underline{a},\underline{b},R_1 ,R_2,\F'} ^{\ast}=
 \widetilde{N}_{\underline{a},\underline{b},R_1 ,R_2,\F'} ^{\ast}
 \big/ \F'[[u]]^n 
\] 
and 
$d_{\underline{a},\underline{b},R_1 ,R_2} ^{\ast} =
 \dim_{\F'} N_{\underline{a},\underline{b},R_1 ,R_2,\F'} ^{\ast}$. 
We note that 
$\dim_{\F'} N_{\underline{a},\underline{b},R_1 ,R_2,\F'} ^{\ast}$ 
is independent of finite extensions $\F'$ of $\F$. 
By Lemma \ref{equiv}, 
giving 
an element of 
$N_{\underline{a},\underline{b},R_1 ,R_2,\F'} ^{\ast}$ and 
$(\gamma_{1,i})_{0 \leq i \leq
 m_{\underline{a},\underline{b}}}$ 
such that 
$\gamma_{1,0} \neq 0$ and 
$\gamma_{1,i} =0$ if 
$i \notin M_{\underline{a},\underline{b},R_1 ,R_2}$
is equivalent to giving 
a point of 
$\mathscr{GR}_{V_{\F},0,\underline{a},\underline{b},R_1 ,R_2}(\F')$. 
By choosing a basis of 
$N_{\underline{a},\underline{b},R_1 ,R_2,\F} ^{\ast}$ over $\F$, 
we have a morphism 
\[
 f_{\underline{a},\underline{b},R_1 ,R_2} 
 :\mathbb{A}_{\F} ^{\bigl( 
 d_{\underline{a},\underline{b},R_1 ,R_2} ^{\ast} 
 + |M_{\underline{a},\underline{b},R_1 ,R_2}| -1 \bigr)}
 \times 
 \mathbb{G}_{m,\F}
 \to 
 \mathscr{GR}_{V_{\F},0} 
\]
such that 
$f_{\underline{a},\underline{b},R_1 ,R_2} (\F')$ 
is injective and 
the image of 
$f_{\underline{a},\underline{b},R_1 ,R_2} (\F')$ 
is equal to
$\mathscr{GR}_{V_{\F},0,\underline{a},\underline{b},R_1 ,R_2}(\F')$. 
We put 
$d_{\underline{a},\underline{b},R_1 ,R_2} = 
 d_{\underline{a},\underline{b},R_1 ,R_2} ^{\ast} 
 +|M_{\underline{a},\underline{b},R_1 ,R_2}|$. 
Then we have (1) and 
\[
 d_{V_{\F}} =
 \max_{
 \mathscr{GR}_{V_{\F},0,\underline{a},\underline{b},R_1 ,R_2}(\F) 
 \neq \emptyset } 
 \bigl\{ d_{\underline{a},\underline{b},R_1 ,R_2} \bigr\}. 
\]
In this maximum, we consider all 
$(\underline{a},\underline{b}) \in 
 \mathbb{Z}^n \times \mathbb{Z}^n$. 
We have already examined 
$d_{\underline{a},\underline{b},R_1 ,R_2}$ 
for $(\underline{a},\underline{b})$ 
such that $a_1 \geq 0$ and 
$b_1 \leq e$. 
So it suffices to bound 
$d_{\underline{a},\underline{b},R_1 ,R_2}$ 
for $(\underline{a},\underline{b})$ 
such that 
$\max \{-a_1 ,b_1 -e\} > 0$. 

\begin{lem}\label{negative}
If $\max \{-a_1 ,b_1 -e\} > 0$, 
the followings hold: 
\begin{enumerate}
\item[(a)]
In the case $e_1 =0$, 
we have $d_{\underline{a},\underline{b},R_1 ,R_2} \leq ne_0 -1$. 
\item[(b)] 
In the case $1 \leq e_1 \leq p-1$, 
we have $d_{\underline{a},\underline{b},R_1 ,R_2} \leq ne_0$. 
\item[(c)] 
In the case $e_1 = p$, 
we have 
$d_{\underline{a},\underline{b},R_1 ,R_2} \leq ne_0 +[n/2]$. 
\end{enumerate}
\end{lem}

\begin{proof}
We put
\begin{align*}
 S_{\underline{a},\underline{b},R_1 ,R_2 ,1} &=\bigl\{
 (v_1 ,0,\ldots ,0) \in \F((u))^n 
 \bigm| v_1 =u^{-r_1},\\ 
 &\hspace{12.4em} 1 \leq r_1 \leq 
 \min \{
 R_1 -m_{\underline{a},\underline{b}} -1, 
 a_{n} ,e-b_{n} 
 \} \bigr\},\\
 S_{\underline{a},\underline{b},R_1 ,R_2 ,2} &=\Biggl\{
 (0,v_2 ,0, \ldots ,0) \in \F((u))^n 
 \Biggm| v_2 =u^{-r_2}, \\
 &\hspace{7.4em} 1 \leq r_2 \leq 
 \min \biggl\{
 R_2 -\max \{-a_1 ,b_1 -e\},
 \frac{e-a_2}{p},\frac{b_2}{p} 
 \biggr\} \Biggr\},\\ 
 S_{\underline{a},\underline{b},R_1 ,R_2 ,i} &=\Biggl\{
 (0,\ldots ,0,v_i ,0, \ldots ,0) \in \F((u))^n 
 \Biggm|  v_i =u^{-r_i}, \\
 &\hspace{12em} 1 \leq r_i \leq 
 \min \biggl\{
 a_{i-1} ,e-b_{i-1} ,\frac{e-a_i}{p},\frac{b_i}{p} 
 \biggr\} \Biggr\} \\
\intertext{for $3 \leq i \leq n$, and} 
 S_{\underline{a},\underline{b},R_1 ,R_2 ,i,j} &=\Biggl\{
 (0,\ldots ,0,v_i ,v_{i+1},\ldots ,v_{j+1} ,0,\ldots ,0) \in \F((u))^n 
 \Biggm| v_i =u^{-r_i},\ \\ 
 & \hspace{-1.2em}
 \ r_i \leq \min \{a_{i-1} ,e-b_{i-1} \}
 \textrm{ if } i \neq 2,\ 
 r_2 \leq R_2 -\max \{-a_1 ,b_1 -e\} 
 \textrm{ if } i=2, \\ 
 & \hspace{1.15em}
 \ u^{a_{l}} v_{l+1} = 
 u^{b_{l}} \phi(v_{l}) 
 \textrm{ and }
 -\!v_u (v_{l+1} )> 
 \min \{a_{l} ,e-b_{l} \} 
 \textrm{ for } 
 i \leq l \leq j, \\ 
 & \hspace{-4.8em}
 -\!v_u (v_{j+1}) \leq 
 \min \biggl\{
 \frac{e-a_{j+1}}{p},\frac{b_{j+1}}{p} 
 \biggr\} \textrm{ if } j \neq n,\ 
 -v_u (v_1) \leq R_1 -m_{\underline{a},\underline{b}} -1 
 \textrm{ if } j=n 
 \Biggr\} 
\end{align*} 
for $2 \leq i \leq j \leq n$. 
In the above definitions, $v_i$ is on the $i$-th component. 
Then, as in the proof of Lemma \ref{basis}, 
we can check that 
$\bigcup_{i} S_{\underline{a},\underline{b},R_1 ,R_2 ,i} 
 \cup \bigcup_{i,j} S_{\underline{a},\underline{b},R_1 ,R_2 ,i,j}$ 
is an $\F$-basis of $N_{\underline{a},\underline{b},R_1 ,R_2,\F} ^{\ast}$. 
So we have 
$d_{\underline{a},\underline{b},R_1 ,R_2} ^{\ast}= 
 \sum_i |S_{\underline{a},\underline{b},R_1 ,R_2 ,i}| + 
 \sum_{i,j} |S_{\underline{a},\underline{b},R_1 ,R_2 ,i,j}|$. 

We put 
\begin{align*}
 T_{\underline{a},\underline{b},R_1 ,R_2 ,1} &=\bigl\{
 m \in \mathbb{Z} \bigm| 
 \min \{a_{n} ,e-b_{n} \} < pm+a_n -b_n 
 \leq R_1 -m_{\underline{a},\underline{b}} -1 
 \bigr\},\\ 
 T_{\underline{a},\underline{b},R_1 ,R_2 ,2} &=\Biggl\{
 m \in \mathbb{Z} \Biggm| 
 R_2 -\max \{-a_1 , b_1 -e\} 
 < R_2 -pm \\ 
 &\hspace{15.85em}\leq 
 \min \biggl\{ R_2 , 
 \frac{e-a_2 }{p} ,
 \frac{b_2 }{p}
 \biggr\}
 \Biggr\} \\
\intertext{and} 
 T_{\underline{a},\underline{b},R_1 ,R_2 ,i} &=\Biggl\{
 m \in \mathbb{Z} \Biggm| 
 \min \{a_{i-1} ,e-b_{i-1} \} 
 < pm+a_{i-1} -b_{i-1} \\
 & \hspace{14.28em}
 \leq \min \biggl\{
 \frac{e-a_i }{p}, \frac{b_i }{p} 
 \biggr\} \Biggr\} 
\end{align*}
for $3 \leq i \leq n$. 
We note that these definitions for 
$S_{\underline{a},\underline{b},R_1 ,R_2 ,i}$, 
$S_{\underline{a},\underline{b},R_1 ,R_2 ,i,j}$ and 
$T_{\underline{a},\underline{b},R_1 ,R_2 ,i}$ 
in the case $\max \{-a_1 ,b_1 -e\} >0$ 
are compatible with the definitions 
in the case $\max \{-a_1 ,b_1 -e\} \leq 0$, 
if $\max \{-a_1 ,b_1 -e\} =0$. 
So in the following, we can consider 
also the case $\max \{-a_1 ,b_1 -e\} =0$. 
We need to consider this case in the following arguments. 
 
We consider the map
\begin{align*}
 \bigcup_{2 \leq j \leq h-1} 
 S_{\underline{a},\underline{b},R_1 ,R_2 ,j,h-1} 
 \cup \bigl\{ 0\leq i \leq m_{\underline{a},\underline{b}}
 \bigm| n_{2,i} =h \bigr\}
 &\to T_{\underline{a},\underline{b},R_1 ,R_2 ,h} ;\\ 
 (v_{i})_{1 \leq i \leq n} &\mapsto -v_u (v_{h-1}) ,\ 
 i \mapsto r_{2,i,h-1}
\end{align*}
for $3 \leq h \leq n+1$. 
We can easily check that this map is injective 
and that 
\[
 \bigl\{ 0\leq i \leq m_{\underline{a},\underline{b}}
 \bigm| n_{2,i} =2 \bigr\} =
 T_{\underline{a},\underline{b},R_1 ,R_2 ,2}. 
\]
So we have 
$\Bigl( \sum_{2 \leq \i \leq j \leq n} 
 |S_{\underline{a},\underline{b},R_1 ,R_2 ,i,j}| \Bigr)
 + m_{\underline{a},\underline{b}} +1 
 \leq \sum_{1 \leq i \leq n} 
 |T_{\underline{a},\underline{b},R_1 ,R_2 ,i}|$ 
and 
\[
 d_{\underline{a},\underline{b},R_1 ,R_2} \leq 
 d_{\underline{a},\underline{b},R_1 ,R_2} ^{\ast} 
 + m_{\underline{a},\underline{b}} +1 \leq 
 \sum_{1 \leq i \leq n} \bigl( 
 |S_{\underline{a},\underline{b},R_1 ,R_2 ,i}| + 
 |T_{\underline{a},\underline{b},R_1 ,R_2 ,i}| \bigr). 
\]

We take 
$(\underline{a}'',\underline{b}'') 
 \in \mathbb{Z}^n \times \mathbb{Z}^n$ 
and $(R_1 '' ,R_2 '') \in I_{\underline{a}'',\underline{b}''}$  
such that $\max \{-a_1 '' ,e-b_1 ''\} \geq 0$ and 
$\sum_{1 \leq i \leq n} \bigl( 
 |S_{\underline{a}'',\underline{b}'',R_1 '' ,R_2 '' ,i}| +
 |T_{\underline{a}'',\underline{b}'',R_1 '' ,R_2 '' ,i}| \bigr)$
is the maximum. 
We can prove that 
$|T_{\underline{a}'',\underline{b}'',R_1 '' ,R_2 '' ,i}| \leq 1$ 
for all $i \neq 2$ as in the proof 
of Lemma \ref{boundT}. 

We show that 
we may take 
$(\underline{a}'',\underline{b}'') 
 \in \mathbb{Z}^n \times \mathbb{Z}^n$ 
and $(R_1 '' ,R_2 '') \in I_{\underline{a}'',\underline{b}''}$ 
such that 
$0 \leq -a_1 '' =b_1 '' -e\leq p-1$. 
If $-a_1 '' >b_1 ''-e$, then 
we replace 
$b_1 ''$ by $b_1 '' +1$ 
and $R_2 ''$ by $R_2 '' +1$.  
We again have 
$(R_1 '' ,R_2 '') \in I_{\underline{a}'',\underline{b}''}$
after the replacement. 
This replacement increases 
$\sum_{1 \leq i \leq n} \bigl( 
 |S_{\underline{a}'',\underline{b}'',R_1 '' ,R_2 '' ,i}| +
 |T_{\underline{a}'',\underline{b}'',R_1 '' ,R_2 '' ,i}| \bigr)$ 
by $0$ or $1$, 
but by the maximality 
there is no case where 
it increases by $1$. 
Similarly, if $-a_1 '' <b_1 ''-e$, 
we may replace 
$a_1 ''$ by $a_1 '' -1$ and 
$R_2 ''$ by $R_2 '' +1$. 
So we may assume $-a_1 '' =b_1 ''-e$. 

If $-a_1 '' \geq p$ and 
$\min \{b_2 '' /p, (e-a_2 '')/p\} \geq R_2 ''$, 
we replace 
$R_1 ''$ by $R_1 '' -1$ and 
$R_2 ''$ by $R_2 '' +p$. 
By 
\[
 R_2 '' +p \leq \frac{e}{p} +p < 
 e+p \leq e-a_1 ''=b_1 '', 
\]
we again have 
$(R_1 '' ,R_2 '') \in I_{\underline{a}'',\underline{b}''}$
after the replacement. 
This replacement increases 
$\sum_{1 \leq i \leq n} \bigl( 
 |S_{\underline{a}'',\underline{b}'',R_1 '' ,R_2 '' ,i}| +
 |T_{\underline{a}'',\underline{b}'',R_1 '' ,R_2 '' ,i}| \bigr)$ 
by at least $p-2$. 
This is a contradiction. 
So if $-a_1 '' \geq p$, we have 
$\min \{b_2 '' /p, (e-a_2 '')/p\}<R_2 ''$. 
If $-a_1 '' \geq p$, 
we replace $a_1 ''$ by $a_1 '' +p$, 
$b_1 ''$ by $b_1 '' -p$, 
$R_1 ''$ by $R_1 '' -1$ and 
$R_2 ''$ by $R_2 '' -p$. 
We again have 
$(R_1 '' ,R_2 '') \in I_{\underline{a}'',\underline{b}''}$
after the replacement. 
This replacement does not 
change 
$\sum_{1 \leq i \leq n} \bigl( 
 |S_{\underline{a}'',\underline{b}'',R_1 '' ,R_2 '' ,i}| +
 |T_{\underline{a}'',\underline{b}'',R_1 '' ,R_2 '' ,i}| \bigr)$. 
Iterating these replacements, 
we may assume 
$0 \leq -a_1 '' =b_1 '' -e\leq p-1$. 
We already treated the case where 
$-a_1 '' =b_1 '' -e=0$. 
So we may assume 
$1 \leq -a_1 '' =b_1 '' -e\leq p-1$. 
We note that
$|T_{\underline{a}'',\underline{b}'',R_1 '' ,R_2 '' ,2}| \leq 1$ 
in this case. 

Now we can show that 
\begin{enumerate}
 \item[$(A_i ')$] 
if 
$|S_{\underline{a}'',\underline{b}'',R_1 '' ,R_2 '' ,i}| +
 |T_{\underline{a}'',\underline{b}'',R_1 '' ,R_2 '' ,i}| 
 =e_0 +l$ 
for $l \geq 1$,\\ 
then 
$|S_{\underline{a}'',\underline{b}'',R_1 '' ,R_2 '' ,i+1}| +
 |T_{\underline{a}'',\underline{b}'',R_1 '' ,R_2 '' ,i+1}| 
 \leq e_0 +e_1 -pl+1$ 
\end{enumerate}
for $i \neq 1$, and that 
\begin{enumerate}
 \item[$(B_i ')$] 
if 
$|S_{\underline{a}'',\underline{b}'',R_1 '' ,R_2 '' ,i}| +
 |T_{\underline{a}'',\underline{b}'',R_1 '' ,R_2 '' ,i}| 
 =e_0 +1$\\ 
and 
$|S_{\underline{a}'',\underline{b}'',R_1 '' ,R_2 '' ,i+1}| +
 |T_{\underline{a}'',\underline{b}'',R_1 '' ,R_2 '' ,i+1}| 
 = e_0 +e_1 -p+1$,\\ 
then 
$|S_{\underline{a}'',\underline{b}'',R_1 '' ,R_2 '' ,i+2}| +
 |T_{\underline{a}'',\underline{b}'',R_1 '' ,R_2 '' ,i+2}| 
 \leq e_0 -(p-1)e_1 +1$
\end{enumerate}
for $2 \leq i \leq n-1$
as in the proof of 
Lemma \ref{AB}. 
By the same argument, 
we can show that 
\begin{enumerate}
 \item[$(A_1 ')$] 
if 
$|S_{\underline{a}'',\underline{b}'',R_1 '' ,R_2 '' ,1}| +
 |T_{\underline{a}'',\underline{b}'',R_1 '' ,R_2 '' ,1}|
 =e_0 +l$ 
for $l \geq 0$,\\ 
then  
$|S_{\underline{a}'',\underline{b}'',R_1 '' ,R_2 '' ,2}| +
 |T_{\underline{a}'',\underline{b}'',R_1 '' ,R_2 '' ,2}| 
 \leq e_0 +e_1 -pl$, 
\end{enumerate}
and that 
\begin{enumerate}
 \item[$(B_n ')$] 
if 
$|S_{\underline{a}'',\underline{b}'',R_1 '' ,R_2 '' ,n}| +
 |T_{\underline{a}'',\underline{b}'',R_1 '' ,R_2 '' ,n}| 
 =e_0 +1$\\ 
and 
$|S_{\underline{a}'',\underline{b}'',R_1 '' ,R_2 '' ,1}| +
 |T_{\underline{a}'',\underline{b}'',R_1 '' ,R_2 '' ,1}| 
 = e_0 +e_1 -p+1$,\\ 
then 
$|S_{\underline{a}'',\underline{b}'',R_1 '' ,R_2 '' ,2}| +
 |T_{\underline{a}'',\underline{b}'',R_1 '' ,R_2 '' ,2}| 
 \leq e_0 -(p-1)e_1$, 
\end{enumerate} 
using the followings: 
\begin{align*}
 &|S_{\underline{a}'',\underline{b}'',R_1 '' ,R_2 '' ,1}| +
 |T_{\underline{a}'',\underline{b}'',R_1 '' ,R_2 '' ,1}| 
 \leq R_1 -1,\ 
 pR_1 +R_2 = e-2a_1 '',\\ 
 &|S_{\underline{a}'',\underline{b}'',R_1 '' ,R_2 '' ,2}| 
 \leq R_2 +a_1 '',\ 
 1 \leq -a_1 '' \leq p-1 
 \textrm{ and } 
 |T_{\underline{a}'',\underline{b}'',R_1 '' ,R_2 '' ,2}|\leq 1. 
\end{align*} 

Then $(A_i ')$ for all $i$ and 
$(B_i ')$ for $i \neq 1$ 
implies that 
\[
 \sum_{1 \leq i \leq n} \bigl( 
 |S_{\underline{a}'',\underline{b}'',R_1 '' ,R_2 '' ,i}| + 
 |T_{\underline{a}'',\underline{b}'',R_1 '' ,R_2 '' ,i}| \bigr) 
 \leq ne_0 
\] 
in the case $0 \leq e_1 \leq p-2$, and that 
\[
 \sum_{1 \leq i \leq n} \bigl( 
 |S_{\underline{a}'',\underline{b}'',R_1 '' ,R_2 '' ,i}| + 
 |T_{\underline{a}'',\underline{b}'',R_1 '' ,R_2 '' ,i}| \bigr) 
 \leq ne_0 +
 \biggl[ \frac{n}{2} \biggr]
\]
in the case $e_1 =p-1$. 
It remains to eliminate the possibility of equality 
in the case $e_1 =0$. 

We assume that $e_1 =0$ and 
$\sum_{1 \leq i \leq n} \bigl( 
 |S_{\underline{a}'',\underline{b}'',R_1 '' ,R_2 '' ,i}| + 
 |T_{\underline{a}'',\underline{b}'',R_1 '' ,R_2 '' ,i}| \bigr) 
 = ne_0$. 
Then $(A_i ')$ for all $i$ 
implies that 
$|S_{\underline{a}'',\underline{b}'',R_1 '' ,R_2 '' ,i}| + 
 |T_{\underline{a}'',\underline{b}'',R_1 '' ,R_2 '' ,i}| =e_0$ 
for all $i$. 
Now we have 
\[
 e_0 =
 |S_{\underline{a}'',\underline{b}'',R_1 '' ,R_2 '' ,1}| + 
 |T_{\underline{a}'',\underline{b}'',R_1 '' ,R_2 '' ,1}| 
 \leq R_1 -1 
\]
and 
\[
 e_0-1 \leq |S_{\underline{a}'',\underline{b}'',R_1 '' ,R_2 '' ,2}| 
 \leq R_2 +a_1 ''.
\]
This implies 
$e+p-1-a_1 '' \leq pR_1 +R_2$. 
Because $pR_1 +R_2 =e-2a_1 ''$, 
this inequality happens only in the case 
$-a_1 ''=p-1$, and in this case 
the above inequalities become equality. 
So we have 
$e_0 -1 = |S_{\underline{a}'',\underline{b}'',R_1 '' ,R_2 '' ,2}|$ 
and $R_2 =e_0 +p-2$. 
By $|T_{\underline{a}'',\underline{b}'',R_1 '' ,R_2 '' ,2}|=1$, 
we have 
$R_2 \leq \min \{(e-a_2 '')/p,b_2 ''/p\}$. 
So we get 
$a_2 ''\leq e_0 -p(p-2) \leq e_0 -3$, 
but this contradicts 
$|S_{\underline{a}'',\underline{b}'',R_1 '' ,R_2 '' ,3}| 
 \geq e_0 -1$. 
Thus we have eliminated 
the possibility of equality 
in the case $e_1 =0$. 
\end{proof}
The claim $(2)$ follows from 
Lemma \ref{nonnegative} and 
Lemma \ref{negative}. 
\end{proof}

\begin{rem}\label{existence}
By Lemma \ref{stracture}, 
we can check that 
there is $V_{\F}$ satisfying 
the conditions for $M_{\F}$ 
in Proposition \ref{irreducible}. 
\end{rem}

\section{Main theorem}

To fix the notation, 
we recall the definition of the zeta function of 
a scheme of finite type over a finite field.

\begin{defn}
Let $X$ be a scheme of finite type over $\F$. 
We put $q_{\F} =|\F|$. 
The zeta function $Z(X;T)$ of $X$ is defined by 
\[
 Z(X;T)=\exp \Biggl( \sum_{m=1} ^{\infty} 
 \frac{\bigl| X(\F_{q_{\F} ^m} )\bigr| }{m} T^m \Biggr).
\]
Here, 
\[
 \exp \bigl( f(T) \bigr) =\sum_{m=0} ^{\infty} 
 \frac{1}{m!} f(T)^m \in \mathbb{Q} [[T]] 
\]
for $f(T) \in T\mathbb{Q} [[T]]$. 
\end{defn}

\begin{thm}
Let 
$Z(\mathscr{GR}_{V_{\F},0} ;T)$ 
be the zeta function of 
$\mathscr{GR}_{V_{\F},0}$. 
Then the followings are true. 
\begin{enumerate}
\item 
After extending the field $\F$ sufficiently, 
we have 
\[
 Z(\mathscr{GR}_{V_{\F},0} ;T) = 
 \prod_{i=0} ^{d_{V_{\F}}} 
 (1-|\F|^i T)^{-m_i} 
\]
for some $m_i \in \mathbb{Z}$ such that $m_{d_{V_{\F}}} >0$. 
\item 
If $n=1$, 
we have 
\[
 0 \leq 
 d_{V_{\F}} \leq  
 \biggl[ \frac{e+2}{p+1} \biggr]. 
\]
If $n \geq 2$, 
we have 
\[
 0 \leq 
 d_{V_{\F}} \leq 
 \biggl[ \frac{n+1}{2} \biggr]
 \biggl[ \frac{e}{p+1} \biggr]+ 
 \biggl[ \frac{n-2}{2} \biggr] 
 \biggl[ \frac{e+1}{p+1} \biggr]+ 
 \biggl[ \frac{e+2}{p+1} \biggr]. 
\]

Furthermore, each equality in the above inequalities 
can happen for any finite extension $K$ 
of $\mathbb{Q}_p$. 
\end{enumerate}
\end{thm} 

\begin{proof} 
This follows from 
Proposition \ref{onepoint}, 
Proposition \ref{reducible}, 
Proposition \ref{irreducible} and 
Remark \ref{existence}. 
\end{proof}

\end{document}